\documentclass[11pt]{article}

\usepackage{amsmath,amsthm}

\usepackage{amssymb,latexsym}

\usepackage[mathscr]{euscript}

\usepackage{enumerate,color}

\newcommand{\BF}{{\mathbb F}}

\newcommand{\BN}{{\mathbf N}}
\newcommand{\BR}{{\mathbb R}}

\newcommand{\FF}{{\mathcal{F}}}

\newcommand{\EE}{{\mathcal{E}}}

\newcommand{\TT}{{\mathcal{T}}}

\newcommand{\essinf}{\mathop{\mathrm{ess\,inf}}}
\newcommand{\esssup}{\mathop{\mathrm{ess\,sup}}}

\newtheorem{tw}{\bf Theorem}[section]
\newtheorem{stw}[tw]{\bf Proposition}

\newtheorem{wn}[tw]{\bf Corollary}

\theoremstyle{definition}
\newtheorem{df}[tw]{Definition}
\newtheorem{prz}[tw]{\bf Example}

\newtheorem{uw}[tw]{Remark}

\numberwithin{equation}{section}

\begin{document}

\title {Non-semimartingale solutions of reflected BSDEs and applications to Dynkin games}
\author {Tomasz Klimsiak\\
{\small Faculty of Mathematics and Computer Science,
Nicolaus Copernicus University} \\
{\small  Chopina 12/18, 87--100 Toru\'n, Poland}\\
{\small E-mail address: tomas@mat.umk.pl}}
\date{}
\maketitle

\date{}
\maketitle
\begin{abstract}
We introduce a new class of  reflected backward stochastic differential equations with two c\`adl\`ag barriers, which need not satisfy any separation conditions.
For that reason, in general, the solutions are not semimartingales.
We prove  existence,  uniqueness and approximation results for  solutions of equations defined on general filtered probability spaces.
Applications to nonlinear  Dynkin games are given.
\end{abstract}



\footnotetext{This work was supported by Polish National Science Centre
(grant no. 2016/23/B/ST1/01543).}

\section{Introduction}
\label{sec1}

In the present paper, we consider  backward stochastic differential equations (RBSDEs) with two reflecting barriers $L$ and $U$. We assume merely that $L,U$ are adapted c\`adl\`ag processes such that $L^+,U^-$ are of class (D) and $L_t\le U_t$, $t\ge 0$. Because, in general, the barriers $L,U$ do not satisfy Mokobodzki's condition (existence of a special semimartingale between the barriers), treating such equations requires extending the notion of a solution to encompass the case where the first component of the solution is not a semimartingale. One of the main novelty of the paper is that we provide such an extension. We show that it is right in the sense that it coincides with the ``classical" definition (semimartingale solutions) if Mokobodzki's condition holds. Our generalized non-semimartingale solutions are unique under the monotonicity assumption on the generators of the equations. Furthermore, we show that under reasonable assumptions on the terminal condition and the generator the solutions exist and can be  approximated  by a penalization scheme. We also prove some stability result and show that there is  one-to-one correspondence between the solutions and the value processes in some generalized Dynkin games (with nonlinear expectation). Let us also stress that in the paper we consider equations on probability spaces equipped with general filtration satisfying only the usual conditions.
Our motivation for studying such general setting comes from applications to Dynkin games and variational inequalities.

Let us mention here that  in \cite{HH0,HH} Hamad\`ene and Hassani introduced  a definition of solution to RBSDE without Mokobodzki's condition,
and proposed a technique to solve it. Based on this technique and definition, in the literature there appeared  some  further results devoted to RBSDEs without Mokobodzki's condition
but, as in the first papers \cite{HH0,HH},  under rather restrictive assumptions on data (see \cite{BY,EHW,HH0,HH,HHO,HW}). We propose a new definition (similar in spirit) of solution to RBSDE without Mokobodzki's condition,  which enables us to provide a different  technique to prove
its existence, uniqueness, approximation and stability    that works in a general framework described below.

We now describe  the content of the paper and give more information about our motivations.
Let $(\Omega,\FF,P)$ be a complete probability space equipped with  a right-continuous complete filtration $\BF=\{\FF_t, t\ge 0\}$,
and let $T$ be  a (possibly infinite) $\BF$-stopping time.
We assume that we are given an $\FF_T$-measurable  integrable random variable $\xi$,
a function $\Omega\times\BR_+\times\BR\ni (\omega,t,y)\mapsto f(\omega,t,y)\in\BR$ (called generator) which is  progressively measurable with respect to $(\omega,t)$, and  $\BF$-adapted c\`adl\`ag processes $L,U$ such that $L^+,U^-$ are of class (D), $L\le U$ and
\begin{equation}
\label{eq0.0.0}
\limsup_{a\rightarrow \infty}L_{T\wedge a}\le \xi\le \liminf_{a\rightarrow \infty}U_{T\wedge a}.
\end{equation}
Let us recall  (see Appendix) that a   (semimartingale) solution of the reflected backward stochastic differential equation with terminal value $\xi$, generator $f$ and barriers $L$ and $U$  (RBSDE$^T(\xi,f,L,U)$ for short) is a triple $(Y,M,R)$ of $\BF$-adapted  c\`adl\`ag processes such that
$Y$ is of class (D), $M$ is a local martingale with $M_0=0$, $R$ is a predictable process  of finite variation with $R_0=0$, and
\begin{equation}
\label{eq1.1}
dY_t=-f(t,Y_t)\,dt-dR_t+\,dM_t,\quad Y_{T\wedge a}\rightarrow \xi\quad\mbox {as } a\rightarrow \infty,
\end{equation}
\begin{equation}
\label{eq1.1asc}
L\le Y\le U,\quad (Y_{t-}-L_{t-})\,dR^+_t=(U_{t-}-Y_{t-})\,dR^-_t=0.
\end{equation}
The crucial part of the above definition is condition (\ref{eq1.1asc}), the so-called Skorokhod condition or minimality condition, which guarantees uniqueness of solutions.
The  Skorokhod condition involves  the predictable finite variation part  of  $Y$, so the semimartingale structure of a solution is   inherent in  the notion of a solution. Clearly, the fact that each solution $Y$ to  RBSDE with barriers $L$ and $U$  is a semimartingale  forces Mokobodzki's condition  ($Y$ is a semimartingale between $L$ and $U$). One of the main goal
of the paper is to introduce a right definition of a solution to RBSDE in the case were Mokobodzki's condition is not satisfied (we merely assume that $L\le U$). In this case it may happen that every c\`adl\`ag process between the barriers $L$ and $U$  is "nowhere a semimartingale", i.e. there is no nontrivial random interval on which it is a semimartingale.  This forces a complete  change of the notion of a solution.  In particular, to get uniqueness,  we are forced to replace the Skorokhod condition (\ref{eq1.1asc}) by a new minimality condition.

It is worth mentioning that usually in the literature there are considered BSDEs  with  generator $f$ depending also on the martingale part $M$ of a solution.
In the whole paper, we focus on RBSDEs with generator $f$ depending only on $Y$ as  we are interested in the case where $Y$
is not a semimartingale, therefore  the martingale part is not well defined. The  case of RBSDEs with generator $f$ depending on $M$  requires further analysis
and its full examination is beyond the scope of the paper, so  we shall  postpone it for the future papers.

Reflected BSDEs with two  barriers satisfying Mokobodzki's condition were introduced by Cvitanic and Karatzas \cite{CvitanicKaratzas} in the case where the barriers $L,U$ are continuous and their supremums are square-integrable,  the terminal value $\xi$ is square-integrable, the terminal time $T$ is constant and finite, the generator $f$ is Lipschitz continuous  and the underlying filtration $\BF$ is Brownian. Since then the notion of reflected BSDEs was  recognized  as a very useful and important tool having applications to stochastic control, mathematical finance and the variational inequalities theory (see, e.g., \cite{BL,HL,K:PA,K:JEE,KR:JFA,KR:JEE,Lin,RS:EJP} and the references therein). Subsequently, in many papers the  assumptions adopted in \cite{CvitanicKaratzas} were weakened (see, e.g., \cite{BY,GIOQ,GIOOQ,HO,Kl:EJP,Kl:BSM,K:SPA2,K:arx,KRS:SPA,LepeltierMatoussiXu,RS:SPA} and the references therein)
but  in the vast majority  of the papers Mokobodzki's condition was  required to hold (see Remark \ref{uw2.1}).

Reflected BSDEs without Mokobodzki's condition have been much less studied.  In \cite{HH0,HHO} the authors considered RBSDEs on Brownian filtered space
with $L^2$-data, $T$ bounded, and Lipschitz continuous generator $f$.  In \cite{BY} the authors generalized the results  of \cite{HH0,HH} by considering
$L^1$-data and monotone $f$ (with linear growth), however the authors additionally assumed that  barriers are continuous and they required  $\sup_{t\le t}L^+_t, \sup_{t\le T} U^-_t$
to be  in $L^1$ (see also \cite{EHW}, where the authors considered similar framework to \cite{BY} but with $L^p$ data for $p\in (1,2)$).
To the authors' best knowledge the only papers on RBSDEs without Mokobodzki's condition on Brownian-Poisson filtered space are \cite{HH,HW},
where the authors considered $L^2$ data, $T$ bounded, Lipschitz continuous generator $f$. However uniqueness is proved only in \cite{HH}, where it is assumed  additionally that the barriers have only totally inaccessible jumps.
It is worth mentioning here that the main goal of the papers mentioned in this paragraph is to prove the existence of semimartingale solutions to RBSDEs with barriers satisfying
the so called complete separation condition, i.e. 
\begin{equation}
\label{eqi.csc12}
L_t<U_t,\,\, t\in [0,T],\quad L_{t-}<U_{t-},\,\, t\in (0,T].  
\end{equation}
The existence of non-semimartingale solution for RBSDEs with barriers satisfying merely $L\le U$ is hidden
in the proof technique considered in these papers: first step is to find a unique non-semimartingale - the so called {\em local solution} - 
under assumption that $L\le U$;  second step is to show that  under additional condition (\ref{eqi.csc12}) local solution is in fact a semimartingale
and solves RBSDE by means of the  classical definition. As a by-product we get that (\ref{eqi.csc12}) implies Mokobodzki's condition. 
At this point we would like to mention the paper \cite{T}, where the author proposed completely different method of solving RBSDEs 
with barriers satisfying (\ref{eqi.csc12}).

One of the most  important result  proved in \cite{CvitanicKaratzas} concerns the connection between (semimartingale) solutions of RBSDEs and so-called Dynkin games introduced in \cite{DY} and  studied extensively by many authors (see, e.g., \cite{A-N,Bismut1,Dynkin,DY,KQC,LM,Morimoto,Stettner,Zabczyk}). In \cite{CvitanicKaratzas} (see \cite{K:SPA2} for the general setting) it is proved that if $Y$ is the first component of a (semimartingale) solution of  RBSDE$^T(\xi,f,L,U)$, then for any stopping time $\alpha\le T$,
\begin{equation}
\label{eqi.1}
Y_\alpha=\esssup_{\sigma\ge \alpha}\essinf_{\tau\ge \alpha}E\Big(\int_\alpha^{\tau\wedge\sigma}f(r,Y_r)\,dr+L_\sigma\mathbf{1}_{\sigma<\tau}
+U_\tau\mathbf{1}_{\tau\le\sigma<T}+\xi\mathbf{1}_{\sigma=\tau=T}|\FF_{\alpha}\Big).
\end{equation}
Recently it was proved in \cite{DQS} for Brownian-Poisson filtration and $L^2$-data (see also \cite{BY} for Brownian filtration and $L^1$-data)  that
under some conditions on $f$  the above equality may be
equivalently stated as
\begin{equation}
\label{eqi.3}
Y_\alpha=\esssup_{\sigma\ge \alpha}
\essinf_{\tau\ge \alpha}\EE^f_{\alpha,\tau\wedge\sigma}
(L_\sigma\mathbf{1}_{\sigma<\tau}
+U_\tau\mathbf{1}_{\tau\le\sigma<T}+\xi\mathbf{1}_{\sigma=\tau=T}),
\end{equation}
where $\EE^f$ is the nonlinear $f$-expectation  introduced  by
Peng \cite{Peng1} (see also \cite{Peng2}). In \cite{EQ} it was
shown that the theory of nonlinear pricing systems has  wide
applications in mathematical finance. When $f=0$,  (\ref{eqi.3})
reduces to  the classical Dynkin game, and when  $f\neq0$,  it is
called a generalized Dynkin game (see \cite{DQS}).

Assume  that $Y$ is a solution to (\ref{eqi.1}) or (\ref{eqi.3}).
Here arises a natural question whether $Y$ is the first component
of a solution to some  reflected BSDE. In general, the answer is
``no", because if $f= 0$ and $L=U$, then from (\ref{eqi.3}) it
follows that $Y=L$. Hence, since we only assume that  $L^+$ is  a
c\`adl\`ag process of class (D), the process $Y$ need not be a
semimartingale. On the other hand,  by the classical  definition of a solution to (\ref{eq1.1}) and (\ref{eq1.1asc}), the first component $Y$ of the solution
is a  semimartingale. We see that to obtain a
one-to-one correspondence between solutions of RBSDEs and
solutions to Dynkin games requires an extension of the notion of a
solution to RBSDE.

The need of extending the notion of reflected BSDEs also arises in
the problems of approximation of the value process in Dynkin
games.  Recall that there are basically two methods of solving
RBSDEs with two reflecting barriers (or solving the related Dynkin
game problem). The first one consists in solving the following
system of optimal stopping problems introduced in
\cite{Bismut1,Bismut2}:
\[
\left\{
\begin{array}{ll} Y^1_t=\esssup_{t\le\tau\le T}E(Y^2_\tau\mathbf{1}_{\tau<T}+\int_t^\tau f(r)\,dr+L_\tau\mathbf{1}_{\tau<T}+\xi\mathbf{1}_{\tau=T}|\FF_t),
\smallskip\\
 Y^2_t=\esssup_{t\le\tau\le T}E(Y^1_\tau\mathbf{1}_{\tau<T}-U_\tau\mathbf{1}_{\tau<T}|\FF_t),
\end{array}
\right.
\]
where  $f$ is independent  of $Y$.
Putting  $Y=Y^1-Y^2$, we obtain  a solution of the  linear
RBSDE$^T(\xi,f,L,U)$. Next, by a fixed point argument, one can
obtain the existence of a solution in the  nonlinear case. Note
that  the above methods always leads to a semimartingale solution,
i.e. $Y$ is a semimartingale, because $Y^1, Y^2$ are supermartingales. The second method is the so-called
penalty method. It is known (see, e.g., \cite{K:SPA2}) that if
Mokobodzki's condition is satisfied, then
under some assumptions on the data, the first component $Y^n$  of
the solution $(Y^n,M^n)$ of  the BSDE
\begin{align}
\label{eqi.5.5} Y^n_t&=\xi+\int_t^T f(r,Y^n_r)\,dr
+n\int_t^T(Y^n_r-L_r)^-\,dr \nonumber\\
&\quad-n\int_t^T(Y^n_r-U_r)^+\,dr-\int_t^T\,dM^n_r,\quad t\in[0,T],
\end{align}
converges as $n\rightarrow\infty$ to a process $Y$ being the
first component of the (semimartingale)  solution to
RBSDE$^T(\xi,f,L,U)$. The question arises whether $\{Y^n\}$
converges if we omit  the assumption of the existence of a special
semimartingale between the barriers. Secondly, if the answer is
``yes", what kind of equation does  the limit process solve? The
problem is rather subtle. It is worth noting here that the penalty
method had been applied to  Dynkin games problems much before the
notion of BSDEs was introduced (see
\cite{Robin,Stettner,Stettner3,SZ}). From the results of Stettner
\cite{Stettner3} (see also \cite{Stettner,Stettner1} for the
Markovian case) it follows (see Remark \ref{uw.stt} for details) that in
the linear case, under some additional  assumptions on the barriers, the  solutions of (\ref{eqi.5.5}) can converge to
a solution of (\ref{eqi.1}) without  Mokobodzki's condition.
Part of our results may be viewed as a far reaching generalization
of Stettner's results on approximation of the value process in
Dynkin games.

As explained above, to show the one-to-one correspondence  between
solutions to RBSDEs and solutions of the  generalized Dynkin
problem (\ref{eqi.1}) or (\ref{eqi.3}) and its approximation by a penalization scheme, we
find ourselves forced to introduce a  new definition of a solution
in which we do not require that the process $Y$ is a
semimartingale. Aiming for this, in the paper we consider
the notion of a local solution to  RBSDE$^T(\xi,f,L,U)$
on a progressively measurable set $A\subset [0,T]\times \Omega$. By local solution we mean an $\BF$-adapted c\`adl\`ag process $Y$ of class (D)
such that $Y_{T\wedge a}\rightarrow \xi$ and for every random interval $[[\alpha,\beta]]\subset A$, where $\alpha\le \beta$
are stopping times, (\ref{eq1.1}) and (\ref{eq1.1asc}) hold on $[[\alpha,\beta]]$.
In this notion it is hidden that $Y$ must be a
special semimartingale locally on $A$, i.e.  on every random interval $[[\alpha,\beta]]\subset A$.

One of the main achievement of the paper  is the observation that  for a given $\mathbb F$-adapted c\`adl\`ag processes $L,U$   such that $L\le U$   there always exists  a family $\{A_\tau,\tau\in\mathcal T\}$ of progressively measurable sets  such that Mokobodzki's condition is satisfied locally on each set $A_\tau$ and there exists at most one  c\`adl\`ag process $Y$ of class (D)
which  solves RBSDE$^T(\xi,f,L,U)$ locally on each $A_\tau$ for every $\tau\in\TT$. We call $Y$ the solution to RBSDE$^T(\xi,f,L,U)$. More precisely, to define a solution, we
first  set
\[
\gamma_\tau=\inf\{\tau<t\le T:L_{t-}=U_{t-}\}\wedge\inf\{\tau\le t\le T: L_t=U_t\}\wedge T
\]
and
\[
\Lambda_\tau=\{L_{\gamma_\tau-}=U_{\gamma_\tau-}\}\cap\{\tau<\gamma_\tau<\infty\}.
\]
We call the family $\{(\gamma_{\tau},\Lambda_{\tau}), \tau\in\mathcal T\}$  the $\ell$-system associated with $L,U$. We then say that a c\`adl\`ag process $X$ is a special $\ell$-semimartingale,  if  for every stopping time $\tau\le T$,  $X$ is a special semimartingale locally on the set $A_\tau$ defined by
\[
A_\tau(\omega):=[\tau(\omega),\gamma_\tau(\omega)\}=
 \left\{
\begin{array}{ll} [\tau(\omega),\gamma_\tau(\omega)], &\omega\notin\Lambda_\tau,
\smallskip\\ {[\tau(\omega),\gamma_\tau(\omega))}, &\omega\in\Lambda_\tau.
\end{array}
\right.
\]
By a solution of RBSDE$^T(\xi,f,L,U)$ we mean a pair $(Y,\Gamma)$ of $\mathbb{F}$-adapted c\`adl\`ag processes such that $Y, \Gamma$ are special $\ell$-semimartingales  satisfying
\begin{equation}
\label{eq1.7}
\begin{cases}
Y_t=Y_{T_a}+\int_t^{T_a}f(r,Y_r)\,dr+\Gamma_{T_a}-\Gamma_t, \quad  t\in [0,T_a],\,\,a\ge0,
\medskip\\
L_t\le Y_t\le U_t,\quad  t\in [0,T_a],\, a\ge 0,\medskip\\
\int_{[\tau,\gamma_\tau\}}(Y_{r-}-L_{r-})\,d\Gamma^{v,+}_r(\tau)
=\int_{[\tau,\gamma_\tau\}}(U_{r-}-Y_{r-})\,d\Gamma^{v,-}_r(\tau)=0,\medskip\\
Y_{T_a}\rightarrow \xi\quad\mbox {a.s.  as } a\rightarrow \infty.
\end{cases}
\end{equation}
In (\ref{eq1.7}), $T_a=T\wedge a$, $\Gamma^v(\tau)$ is the (unique) predictable finite
variation process on $[\tau,\gamma_\tau\}$  from the Doob-Meyer decomposition of $\Gamma$ on
$[\tau,\gamma_\tau\}$:
\[
\Gamma_t=\Gamma_\tau+\Gamma_t^v(\tau)+\Gamma_t^m(\tau),\quad t\in [\tau,\gamma_\tau\},
\]
where $\Gamma^m(\tau)$ is a  local martingale on $[\tau,\gamma_\tau\}$ with $\Gamma^m_\tau(\tau)=0$. 
At first glance  the above definition of a solution, making use of the family $\{A_{\tau}\}$, may seem a bit intricate.
In Examples \ref{dex.1}--\ref{dex.3} we shall show that to get uniqueness of solutions it  cannot be simpler in the sense that in general we cannot replace \{$A_{\tau}\}$
by a single progressively measurable set or by intervals $[[\tau,\gamma_\tau[[$ or $[[\tau,\gamma_\tau]]$ defined independently of $Y$ through $L$ and $U$.

From now on, solutions of RBSDEs in the sense of
(\ref{eq1.1}), (\ref{eq1.1asc}) will be called semimartingale
solutions, and the solutions in the generalized sense will be
called non-semimartingale solutions or simply solutions.
In Section \ref{sec4}, we show 
that if $(Y,\Gamma)$ satisfies (\ref{eq1.7})
and there exists a special semimartingale between the barriers $L$ and $U$,  then $Y,\Gamma$ are special semimartingales and the triple $(Y,\Gamma^v,\Gamma^m)$, where $\Gamma^v$ (resp. $\Gamma^m$) is the predictable finite variation part (resp.  martingale part)  from the Doob-Meyer decomposition of $\Gamma$ on $[0,T]$, is
a semimartingale solution of RBSDE$^T(\xi,f,L,U)$.

In the context of reflected  BSDEs, the concept of local solutions (somewhat different from ours, since we consider here arbitrary progressively measurable set $A$)  was considered for the first time by Hamad\`ene and Hassani 
in \cite{HH0} (see also \cite{HH,HHO}) and also, as it is in the present paper, this concept was used to define a unique solution to RBSDE without Mokobodzki's condition.
The definition proposed in \cite{HH0} is similar in spirit to ours but instead of intervals $[\tau,\gamma_\tau\}$ the authors in \cite{HH0}
consider shorter closed intervals $[[\tau,\delta_\tau]]$, with $\delta_\tau$ defined through $Y$. 
However, the main difference between our definition and the one considered in \cite{HH0} is that here intervals  $[\tau,\gamma_\tau\}$  are defined
{\em independently} of $Y$.  For  this "independence" we pay a small price by using slightly more complex  intervals $[\tau,\gamma_\tau\}$  that are  randomly closed/open from the right. In  return, based on this new definition,   we may  provide  new  techniques  thanks to which we can  prove
in a simple way the existence, uniqueness and approximation results in our general framework.

We now describe the content of the paper. 
In Sections \ref{sec3} and \ref{sec4}, we introduce in more detail the notion of non-semimartingale solutions of RBSDEs, and  we prove our main results on existence, uniqueness and approximation of  solutions. 
First, we show that  under the assumption that $y\mapsto f(t,y)$ is nonincreasing a comparison theorem for solutions to
RBSDEs 
holds true. It  implies uniqueness of solutions.  Moreover,  we prove stability of solutions, 
i.e. we show that if $(Y^i,\Gamma^i)$, $i=1,2$, are  solutions of
RBSDE$^{T}(\xi^i,f^i,L^i,U^i)$, then
\begin{align}
\label{eq1.9}
\|Y^1-Y^2\|_{1;T}&\le E|\xi^1-\xi^2|
+E\int_0^{T}|f_1(t,Y^2_t)-f_2(t,Y^2_t)|\,dt\nonumber\\
&\quad+\|L^1-L^2\|_{1;T}+\|U^1-U^2\|_{1;T},
\end{align}
where $\|Y\|_{1;T}=\sup_{\tau\le T, \tau<\infty}E|Y_\tau|$.
To show  the existence of a solution, we additionally impose some integrability
conditions on $f$. In the paper, we assume that
\begin{equation}
\label{eq1.12}
\int_0^T|f(t,y)|\,dt<\infty\quad\mbox{ for every }y\in\BR
\end{equation}
and there exists a c\`adl\`ag process $S$ being a  difference of
supermartingales of class (D) such that
\begin{equation}
\label{eq1.11}E\int_0^T|f(t,S_t)|\,dt<\infty.
\end{equation}
The second condition is commonly used in  the literature with
$S=0$. Both conditions are the minimal known conditions ensuring the
existence of solutions of BSDEs with no reflection (condition
(\ref{eq1.11}) is necessary when $f$ is positive).
We prove that if   the function $y\mapsto f(t,y)$ is continuous and nonincreasing, and moreover, $f$ satisfies  (\ref{eq1.12}) and  (\ref{eq1.11}),  then there exists a unique solution
$(Y,\Gamma)$ of RBSDE$^T(\xi,f,L,U)$. We also show that under these assumptions  for every
strictly positive bounded $\mathbb F$-progressively measurable
process $\eta$ such that
\[
E\int_0^T\eta_t(S_t-L_t)^-\,dt+E\int_0^T\eta_t(S_t-U_t)^+\,dt<\infty
\]
there exists a unique solution to the following penalized BSDE
\begin{align}
\label{eqi.8}
Y^n_t&=\xi+\int_t^T f(r,Y^n_r)\,dr
+n\int_t^T\eta_r(Y^n_r-L_r)^-\,dr\nonumber\\
&\quad-n\int_t^T\eta_r(Y^n_r-U_r)^+\,dr-\int_t^T\,dM^n_r,
\end{align}
and for every $a\ge 0$,
\[
Y^n_t\rightarrow Y_t,\quad \int_0^tn(Y^n_r-L_r)^-\,dr
-\int_0^tn(Y^n_r-U_r)^+\,dr-M^n_t\rightarrow \Gamma_t,\quad t\in
[0,T_a].
\]
In the case where $T$ is bounded, one can take $\eta\equiv1$, so
(\ref{eqi.8}) reduces to the  usual penalization scheme
(\ref{eqi.5.5}). Moreover, we show that if
\begin{equation}
\label{eqi.2.3}
^pL\ge L_-\,,\qquad ^pU\le U_-\,,
\end{equation}
where $^pL$ (resp. $^pU$) is the predictable projection of $L$
(resp. $U$), then
the convergence of $\{Y^n\}$ is uniform in probability on compact
subsets of $\BR_+$ (the so-called ucp convergence).

In Section \ref{sec5}, under the additional assumption that $L,U$ are of class (D)  (and not merely $L^{+}$ and $U^{-}$), we study connections of RBSDEs with  Dynkin games
and nonlinear expectation. We show that if $Y$ is a solution of (\ref{eqi.1}), then $Y$ is
the first component of a solution of RBSDE$^T(\xi,f,L,U)$, and
conversely, if $(Y,\Gamma)$ is a solution of RBSDE$^T(\xi,f,L,U)$
and $E\int_0^T|f(r,Y_r)|\,dr<\infty$, then $Y$ is a solution to
(\ref{eqi.1}). 
We also prove that if (\ref{eqi.2.3}) is satisfied, then $(\sigma_\alpha^*,\tau_\alpha^*)$ defined by
\begin{equation}
\label{eqi.2} \sigma^*_\alpha=\inf\{t\ge\alpha: Y_t=L_t\}\wedge T,
\qquad \tau^*_\alpha=\inf\{t\ge\alpha:Y_t=U_t\}\wedge T
\end{equation}
is a saddle point for (\ref{eqi.1}). 
Moreover, the  process $Y+\int_\alpha^\cdot f(r,Y_r)\,dr$ is a
uniformly integrable martingale on the closed interval
$[\alpha,\sigma^*_\alpha\wedge\tau^*_\alpha]$.

We next generalize  the notion of the nonlinear $f$-expectation introduced
in \cite{Peng1} for Brownian filtration and square
integrable data, and then extended in
\cite{QS} to the case of filtration generated
by Brownian motion and an independent Poisson random measure, and we show that
\begin{equation}
\label{eq.eqx2}
(Y,\Gamma)\mbox{ is a solution of RBSDE$^T(\xi,f,L,U)$
\quad iff\quad $Y$ satisfies (\ref{eqi.3})}.
\end{equation}
Let us stress here that  (\ref{eq.eqx2}) holds true although in general the integral
$E\int_0^T|f(r,Y_r)|\,dr$ may be infinite. Furthermore, we show
that under (\ref{eqi.2.3}) the pair (\ref{eqi.2}) is a saddle
point for the  generalized Dynkin game (\ref{eqi.3}).

Finally,
in Section \ref{sec7},  we show that if $L$ and $U$ are of class (D), then there is  a solution of
RBSDE$^T(\xi,f,L,U)$ even if we drop condition (\ref{eq1.11}). Unfortunately, we
do not know whether  it is a limit of some penalization scheme. Nevertheless, this result is
 interesting because it implies that there exist data $\xi,f$ ($\xi\in L^1$ and $f$ is continuous and nonincreasing with respect to $y$ and satisfies (\ref{eq1.12})) such that  there is no solution of class (D) to BSDE$^T(\xi,f)$ but there exists a solution to RBSDE$^T(\xi,f,L,U)$ for all c\`adl\`ag barriers of class (D) satisfying (\ref{eq0.0.0}) (see Example \ref{prz4.1.09}).

In Section \ref{sec2} (Appendix), we collect results on semimartingale solutions to RBSDEs on general filtered space.
We also  extend the results of \cite{K:SPA2} to the case of arbitrary, possibly unbounded terminal time $T$.

\section{Notation and standing assumptions}
\label{secN}

In the paper, $(\Omega,\FF,P)$ is a complete probability space equipped with  a right-continuous complete filtration  $\BF=\{\FF_t, t\in[0,\infty]\}$ with $\FF_\infty=\bigvee_{t\ge 0}\FF_t$.
We assume that we are given
a function
\[
\Omega\times\BR_{+}\times\BR\ni (\omega,t,y)\mapsto f(\omega,t,y)\in\BR
\]
which is  $\BF$-progressively measurable with respect to $(\omega,t)$ for every $y\in\BR$.
Throughout the paper, unless stated otherwise, we assume that  $T$ is an $\BF$-stopping time, i.e. $T:\Omega\to [0,\infty]$ is a random variable, and $\{T\le t\}\in\FF_t$ for any $t\ge 0$,
 and  $\xi$  is an  $\FF_T$-measurable random variable.
We also assume that   $L,U$ are  $\BF$-adapted c\`adl\`ag processes  such that $L^+,U^-$ are of class (D), and moreover, $L_t\le U_t$, $ t\in [0,T\wedge a]$, $a\ge 0$, and
\begin{equation}
\label{eq0.0.0abc}
\limsup_{a\rightarrow \infty}L_{T\wedge a}\le \xi\le \liminf_{a\rightarrow \infty}U_{T\wedge a}.
\end{equation}
In the sequel, for given processes $X,Y$, we write $X\le Y$ if  $X_t\le Y_t,\, t\in [0,T\wedge a]$, $a\ge 0$ a.s.
Let us recall that and $\BF$-adapted process $K$
is called an increasing process if its trajectories are nonincreasing a.s. For a given finite variation process $V$, we denote by  $V^+, V^-$ the unique
increasing processes from the Hahn-Jordan decomposition of $V$ ($V=V^+-V^-$).

Let $\alpha,\beta$ be two  stopping times such that $\alpha\le\beta$.
We say that an $\BF$-progressively measurable process $Y$ is of class (D) on $[\alpha,\beta]$
if the family $\{Y_\tau, \alpha\le\tau\le\beta, \tau<\infty\}$ is uniformly integrable. We set
\begin{equation}
\label{eq2.dn1}
\|Y\|_{1;\alpha,\beta}=\sup_{\alpha\le \tau\le\beta, \tau<\infty}E|Y_\tau|,\qquad \|Y\|_{1;\beta}=\|Y\|_{1;0,\beta}.
\end{equation}
We say that a nonincreasing sequence of stopping times $\{\tau_k\}$
is a chain on $[\alpha,\beta]$ if $\alpha\le\tau_k\le\beta$, $k\ge 1$, and  the set $\{k\ge1:\tau_k< \beta\}$ is finite a.s.

We denote by  $\TT$  the set of all $\BF$-stopping times $\tau$ such that $\tau\le T$, and for given $\alpha\in\mathcal T$, we denote by $\mathcal T_\alpha$ the set of all $\tau\in \mathcal T$ such that $\tau\ge \alpha$.
For a stopping time $\sigma$ and  $\Lambda\in\FF_\sigma$, we set
\[
\sigma_\Lambda(\omega)= \left\{
\begin{array}{ll}\sigma(\omega), &\omega\in\Lambda,
\smallskip\\
 \infty, &\omega\notin\Lambda.
\end{array}
\right.
\]
It is well known that $\sigma_\Lambda$ is a stopping time.

For  stopping times $\alpha\le\beta$, we denote by $[[\alpha,\beta]]$  the random interval defined as
\[
[[\alpha,\beta]]=\{(t,\omega)\in [0,\infty)\times\Omega: \alpha(\omega)\le t\le \beta(\omega)\}.
\]
We put  $[[\alpha]]:= [[\alpha,\alpha]]$, $]]\alpha,\beta]]:= [[\alpha,\beta]]\setminus [[\alpha]]$, $[[\alpha,\beta[[:= [[\alpha,\beta]]\setminus [[\beta]]$, $]]\alpha,\beta[[:=[[\alpha,\beta]]\setminus([[\alpha]]\cup[[\beta]])$.

In the sequel, for a given progressively measurable set $A\subset \mathbb R^+\times\Omega$, we say that some property holds locally on $A$ if it holds on $[[\alpha,\beta]]\subset A$ for every $\alpha,\beta\in \TT$ such that $\alpha\le\beta$.

In the paper, we use frequently the notions and results concerning semimartingale solutions to BSDEs and  reflected BSDEs which are collected in  Appendix.
Especially, we use the notions of solutions to  BSDE$^{\alpha,\beta}(\xi,f)$,
$\mbox{\underline{R}}$BSDE$^{\alpha,\beta}(\xi,f,L)$, $\overline{\rm{R}}${\rm BSDE}$^{\alpha,\beta}(\hat\xi,f,U)$,
RBSDE$^{\alpha,\beta}(\xi,f,L,U)$ for given $\alpha,\beta\in\mathcal T,\, \alpha\le \beta$. We also use the   convention that  BSDE$^T$ stands for BSDE$^{0,T}$.

\section{Reflected BSDEs with bounded terminal time}
\label{sec3}

In this section, we assume that $T$ is a bounded $\BF$-stopping time.

\begin{df}
Let $A\subset \mathbb R_+\times\Omega$ be a progressively measurable set.
We say that a c\`adl\`ag progressively measurable process $Y$ is a solution of RBSDE$^T(\xi,f,L,U)$ locally on $A$ if   $Y_T=\xi$, and for all $\alpha,\beta\in \TT$  such that $\alpha\le\beta$ and  $[[\alpha,\beta]]\subset A$,  $Y$ is a semimartingale  solution of RBSDE$^{\alpha,\beta}(Y_\beta,f,L,U)$.
\end{df}

In the present paper, we  assume merely that $L\le U$, so in general  the barriers do not satisfy
Mokobodzki's condition. For that reason, solutions to reflected  BSDEs need not be semimartingales. Therefore, as already explained in Introduction,
to deal with equations with such barriers requires the introduction of
new definition of a solution.
The problem with new definition is rather subtle. Consider some reflected equation with
c\`adl\`ag barriers $L$ and $U$ such that $L\le U$. The first (naive) idea to solve it is the following. We find a progressively measurable set $A$ (the bigger the better) on which Mokobodzki's  condition is locally satisfied. We then solve the equation  locally on $A$ and get the solution by aggregation local solutions on $A$ and
by putting some natural value on the set $A^c$.  For instance,  we solve  locally the equation on the  progressively measurable set $A=\{L<U\}$ and next we  put $Y=L=U$ on the set $\{L=U\}$. Unfortunately, this approach fails. This follows from Examples \ref{dex.1}--\ref{dex.3} given below.
\begin{enumerate}
\item[(a)] In Example \ref{dex.1}, we show that  there exist $\mathbb F$-adapted c\`adl\`ag barriers $L,U$ ($L\le U$) such that there is no c\`adl\`ag process between the barriers such that
it is a semimartingale  locally on the set $\{L<U\}$. This shows that we cannot expect that a solution to RBSDE is a semimartingale locally on the set $\{L<U\}$ and, as a consequence, we cannot solve  RBSDE locally on the set $\{L<U\}$.

\item[(b)] By Remark \ref{uw2.1}, for any  $\mathbb F$-adapted  c\`adl\`ag barriers $L,U$ ($L\le U$) Mokobodzki's condition is satisfied locally on the set $\{L<U\}\cap \{L_-<U_-\}$, so it is possible to solve RBSDE locally on the set
$\{L<U\}\cap \{L_-<U_-\}$. In Example \ref{dex.2}, we show that there may be infinitely many processes $Y$ of class (D)
which solve some linear RBSDE locally on the set $\{L<U\}\cap \{L_-<U_-\}$.

\item[(c)]  By (a), the set $\{L<U\}$ is too big to satisfy locally Mokobodzki's condition. On the other hand, although Mokobodzki's condition is satisfied locally on the set  $\{L<U\}\cap \{L_-<U_-\}$, (b) shows that  it  is to small to get uniqueness. In Example \ref{dex.3}, we  show that it may happen that there is no progressively measurable set $A$ such that
\[
\{L<U\}\cap \{L_-<U_-\}\,\, \subset\,\, A\,\,\subset\,\, \{L<U\}
\]
having the property that  Mokobodzki's condition is satisfied locally on $A$ and we get uniqueness by solving RBSDE locally on $A$.

\end{enumerate}

\begin{prz}
\label{dex.1}
Let $\Omega=\BR$, $T=2$ and $\FF_t=\{\emptyset,\Omega\}$, $\BF=\{\FF_t, t\in [0,T]\}$. We set $f\equiv 0$, $\xi\equiv0$, and
\[
L_t=(1-t)\cos(\frac{\pi}{1-t})\mathbf{1}_{[0,1)}(t), \quad t\in[0,T],
\]
\[
U_t=(L_t+\frac12(1-t))\mathbf{1}_{[0,1)}(t)+\mathbf{1}_{[1,2]}(t),\quad t\in[0,T].
\]
Since the filtration $\BF$ is trivial, a process $Y$ is an $\BF$-semimartingale if and only if
it is a process of finite variation.
Of course, any solution of the problem RBSDE$^T(\xi,f,L,U)$ has to satisfy $L\le Y\le U$. In particular, putting $t_n=(n-1)/n$, we have
\[
\frac{1}{2n}=L_{t_{2n}}\le Y_{t_{2n}},\quad Y_{t_{2n+1}}\le U_{t_{2n+1}}=\frac{-1}{2n+1}+\frac{1}{4n+2}\,.
\]
Observe that $\{L=U\}=\emptyset$ and
\[
\mbox{Var}_{[0,1]}(Y)\ge \sum_{n=2}^\infty |Y_{t_{2n}}-Y_{t_{2n+1}}|\ge \sum_{n=2}^\infty\frac{1}{2n+1}=\infty,
\]
so $Y$ is not a semimartingale. Observe that $L_{-1}=U_{1-}$, so the above example is not in contradiction with Remark \ref{uw2.1}.
\end{prz}

\begin{prz}
\label{dex.2}
We define $\Omega,T$ and $\BF$ as in Example \ref{dex.1}.  Let
\[
L_t=-t\mathbf{1}_{[0,1)}(t)+(t-1)\sin{[\pi(t-1)]}\mathbf{1}_{[1,2]}(t),\quad t\in [0,2],
\]
\[
U_t=t\mathbf{1}_{[0,1)}(t)+(t-1)\sin{[\pi(t-1)]}\mathbf{1}_{[1,2]}(t),\quad t\in [0,2].
\]
Observe that $\{L< U\}\cap\{L_-< U_-\}=\{L< U\}=[0,1)$.
Let $r\in (0,1)$ and
\[
Y^r_t=t\mathbf{1}_{[0,r)}(t)+r\mathbf{1}_{[r,1)}(t)+(t-1)\sin{[\pi(t-1)]}\mathbf{1}_{[1,2]}(t),\quad t\in [0,2].
\]
It is easy to verify that for every $r\in (0,1)$ the process $Y^r$ is a special semimartingale of class (D) with $Y^r_T=0$, and that $Y^r$ is a solution of RBSDE$^{a,b}(Y_b,0,L,U)$ for every $a,b\in [0,1)$ with  $a\le b$.
\end{prz}

\begin{prz}
\label{dex.3}
We define $\Omega,T$ and $\BF$ as in Example \ref{dex.1}.  We set
\[
L^0_t= 1-\sum_{n=1}^\infty(t-\frac{1}{n+1})\mathbf{1}_{[\frac{1}{n+1},\frac1n)}(t),\quad t\ge 0,
\]
\[
U^0_t= 1+\sum_{n=1}^\infty(t-\frac{1}{n+1})\mathbf{1}_{[\frac{1}{n+1},\frac1n)}(t),\quad t\ge 0,
\]
and then,  for $t\in [0,2]$ we set
\[
L_t=(1+(t+1)\cos{\frac{\pi}{1+t}})L^0_{1+t}-\mathbf{1}_{[0,1)}(t),
\]
\[
U_t=(1+(t+1)\cos{\frac{\pi}{1+t}})U^0_{1+t}+\mathbf{1}_{[0,1)}(t).
\]
Observe that $\{L< U\}\cap \{L_-< U_-\}=[0,2]\setminus N$, where $N=\{1\}\cup\{1+\frac1n,\, n\ge 2\}$.
From Example \ref{dex.2} it follows that for every $a\in N$, if there exists a c\`adl\`ag progressively measurable process $Y$ of class (D)
with $Y_T=\xi$ such that $Y$ solves RBSDE$^T(0,0,L,U)$ locally on $[0,T]\setminus \{a\}$, then there are infinitely many processes with these properties. Therefore the only  extension of $\{L< U\}\cap \{L_-< U_-\}$ ensuring uniqueness of $Y$ is the whole interval
 $[0,2]$. However, from the construction of $L,U$ it follows that each process  $Y$ lying between $L$ and $U$ is of infinite variation on $[0,2]$, so it is not a special semimartingale.
\end{prz}

\subsection{Definition of a solution}

For  $\tau\in\TT$, we define the stopping time $\dot\gamma_\tau$ by
\[
\dot\gamma_\tau=\inf\{\tau<t\le T: L_{t-}=U_{t-}\}\wedge\inf\{\tau\le t\le T:L_t=U_t\},
\]
and then we set
\begin{equation}
\label{eq2.1}
\gamma_\tau=\dot\gamma_\tau\wedge T,\qquad \Lambda_\tau=\{L_{\gamma_\tau-}=U_{\gamma_\tau-}\}\cap\{\tau<\gamma_\tau\}.
\end{equation}
Observe that $\Lambda_\tau\in\FF_{\gamma_\tau-}$ and the stopping time $(\gamma_\tau)_{\Lambda_\tau}$ is predictable since   the sequence $\{\alpha_n:=\inf\{\tau<t<\gamma_\tau: |L_t-U_t|\le\frac1n\}\wedge n$ announces it.
Let $\{\dot\delta^{k}_\tau\}$  ($\dot\delta^{k}_\tau\ge\tau$) be an announcing sequence for $(\gamma_\tau)_{\Lambda_\tau}$ and let $\delta^{k}_\tau=\dot\delta^{k}_\tau\wedge T$. We put
\begin{equation}
\label{eq2.3cdf}
\gamma^{k}_\tau=\delta_\tau^{k}\wedge\gamma_\tau.
\end{equation}
In the whole paper we use the following notation
\[
[\tau(\omega),\gamma_\tau(\omega)\}=
 \left\{
\begin{array}{ll} [\tau(\omega),\gamma_\tau(\omega)], &\omega\notin\Lambda_\tau,
\smallskip\\ {[\tau(\omega),\gamma_\tau(\omega))}, &\omega\in\Lambda_\tau.
\end{array}
\right.
\]
We call the family $\{(\gamma_\tau,\Lambda_\tau), \tau\in\TT\}$  the $\ell$-system associated with $L$ and $U$.
Observe that
\[
[\tau,\gamma_\tau\}=\bigcup_{k\ge 1} [\tau,\gamma^{k}_\tau].
\]
In what follows we also adopt the convention that $[a,a]=[a,a)=\{a\}$.

Let   $\alpha,\beta\in\TT$ ($\alpha\le\beta$). We say that an  $\BF$-adapted process $\Gamma$ is a (local) martingale (resp. (predictable) increasing process) on
$[\alpha,\beta]$ if there exist
a (local) $\BF$-martingale $M$ (resp. a (predictable) increasing $\BF$-adapted process $A$) such that  $\Gamma_t=M_t$, $t\in [\alpha,\beta]$ (resp. $\Gamma_t=A_t$, $t\in [\alpha,\beta]$).

We say that an $\BF$-adapted
process $\Gamma$ is a (local) martingale (resp. (predictable) increasing process) on $[\tau,\gamma_\tau\}$ if it is a (local) martingale (resp. (predictable) increasing process)  on $[\tau,\gamma^{k}_\tau]$ for $k\ge 1$.
\begin{df}
We say that an $\BF$-adapted c\`adl\`ag process $\Gamma$ is an $\ell$-martingale (resp. local $\ell$-martingale) if
it is a martingale (resp. local martingale) on $[\tau,\gamma_\tau\}$ for every $\tau\in\TT$. We say that $\Gamma$
is an $\ell$-semimartingale (resp. special $\ell$-semimartingale) if $\Gamma$ is a semimartingale
(resp. special semimartingale) on $[\tau,\gamma_\tau\}$ for every $\tau\in\TT$.
\end{df}

For a given special $\ell$-semimartingale $\Gamma$,  we denote by $\Gamma^v(\tau)$ (resp. $\Gamma^m(\tau)$)  its predictable finite variation part (resp. local martingale part)
from the Doob-Meyer decomposition on $[\tau,\gamma_\tau\}$. For a process $\Gamma$ and finite $\alpha,\beta\in\TT$ such that $\alpha\le\beta$, we denote by $\int_\alpha^\beta\,d\Gamma_r$ the difference $\Gamma_\beta-\Gamma_\alpha$.

\begin{df}
\label{df.main}
We say that a pair $(Y,\Gamma)$ of $\BF$-adapted c\`adl\`ag process  is a solution of the reflected backward stochastic differential equation on the interval $[0,T]$ with terminal time $\xi$, generator  $f$, lower barrier $L$ and upper barrier $U$ (RBSDE$^T(\xi,f,L,U)$ for short) if
\begin{enumerate}
\item[(a)] $Y$ is of class (D), $\Gamma$ is a special $\ell$-semimartingale,
\item[(b)] $\int_0^T|f(r,Y_r)|\,dr<\infty$ and
\[
Y_t=\xi+\int_t^T f(r,Y_r)\,dr+\int_t^T\,d\Gamma_r,\quad t\in [0,T],
\]
\item[(c)] $L_t\le Y_t\le U_t,\, t\in [0,T]$,
\item[(d)] for every $\tau\in\TT$,
\[
\int_{[\tau, {\gamma_\tau}\}}(Y_{r-}-L_{r-})\,d\Gamma^{v,+}_r(\tau)
=\int_{[\tau,{\gamma_\tau}\}}(U_{r-}-Y_{r-})\,d\Gamma^{v,-}_r(\tau)=0.
\]
\end{enumerate}
\end{df}

\begin{uw}
Of course, in the above definition the process $\Gamma $ is determined by $Y$ through the formula
\[
\Gamma_t=Y_0-Y_t-\int_0^tf(r,Y_r)\,dr,\quad t\in [0,T].
\]
That is why in the whole paper we shall write that a solution of RBSDE is $Y$ and $(Y,\Gamma)$
interchangeably.
\end{uw}

\begin{uw}
Consider the very special case where $L=U$. If $(Y,\Gamma)$ is a solution of RBSDE$^T(\xi,f,L,U)$, then of course
\[
Y_t=L_t,\qquad \Gamma_t=-\int_0^t f(r,L_r)\,dr-L_t+L_0,\quad t\in [0,T],
\]
and $\gamma_\tau=\tau$ for every $\tau\in\TT$.

\end{uw}

\subsection{Existence,  uniqueness and approximation of solutions}

Let us consider the following hypotheses:
\begin{enumerate}
\item[(H1)] $E|\xi|<\infty$  and there exists a c\`adl\`ag process $S$, which is a difference
of supermartingales of class (D), such that $E\int_0^T|f(r,S_r)|\,dr<\infty$.

\item[(H2)]  there exists $\mu\in\BR$ such that for a.e. $t\in [0,T]$,
\[
(y-y')(f(t,y)-f(t,y'))\le\mu |y-y'|^2,\quad y,y'\in\BR,\quad \mbox{a.s.}
\]
\item[(H3)]For a.e. $t\in [0,T]$ the function $y\mapsto f(t,y)$ is continuous a.s.

\item[(H4)] For every $y\in\BR$,  $\int_0^T|f(r,y)|\,dr<\infty$ a.s.
\end{enumerate}

We start with a comparison result.
\begin{tw}
\label{tw.unq}
 Assume that $\xi_1\le\xi_2$, $L^1_t\le L^2_t,\, U^1_t\le U^2_t$, $t\in [0,T]$, and  for a.e. $t\in [0,T]$ we have $f^1(t,y)\le f^2(t,y)$ for all $y\in\BR$. Assume also that $f^1$ satisfies \mbox{\rm(H2)}. Let $(Y^i,\Gamma^i)$, $i=1,2$, be a solution to {\rm RBSDE}$^T(\xi^i,f^i,L^i,U^i)$. Then
\[
Y^1_t\le Y^2_t,\quad t\in [0,T].
\]
\end{tw}
\begin{proof}
Let $\tau\in\TT$ and $(\gamma_\tau^1,\Lambda^1_\tau), (\gamma_\tau^2,\Lambda^2_\tau)$ be defined by (\ref{eq2.1})
but with $L,U$ replaced by $L^1,U^1$ and $L^2,U^2$, respectively. Let $\{\gamma^{1,k}_\tau\},\, \{\gamma^{2,k}_\tau\}$
be the sequences constructed as in (\ref{eq2.3cdf}) but for $\gamma_{\tau}$ replaced by $\gamma^1_\tau$ and $\gamma^2_\tau$, respectively.
By the definition, $Y^i$ is a special semimartingale on $[\tau,\gamma^i_\tau\},\, i=1,2.$ In particular, $Y^1,Y^2$
are special semimartingales on $[\tau,\gamma^{1,k}_\tau\wedge\gamma^{2,k}_\tau]$. By the Tanaka-Meyer
formula and (H2),
\begin{align*}
E(Y^1_\tau-Y^2_\tau)^+&\le
E(Y^1_{\gamma^{1,k}_\tau\wedge\gamma^{2,k}_\tau}-Y^2_{\gamma^{1,k}_\tau\wedge\gamma^{2,k}_\tau})^+\\
&\nonumber\quad +E\int_\tau^{\gamma^{1,k}_\tau\wedge\gamma^{2,k}_\tau}\mbox{sgn}(Y^1_{r-}-Y^2_{r-})\,d(\Gamma^{1,v}_r(\tau)
-\Gamma^{2,v}_r(\tau))\\
&\quad\nonumber
+\mu^+E\int_\tau^{\gamma^{1,k}_\tau\wedge\gamma^{2,k}_\tau}(Y^1_r-Y^2_r)^+\,dr.
\end{align*}
Hence, by condition (d) of Definition \ref{df.main},
\begin{align}
\label{eq3.1}
E(Y^1_\tau-Y^2_\tau)^+&\le
E(Y^1_{\gamma^{1,k}_\tau\wedge\gamma^{2,k}_\tau}-Y^2_{\gamma^{1,k}_\tau\wedge\gamma^{2,k}_\tau})^+
+\mu^+E\int_\tau^{\gamma^{1}_\tau\wedge\gamma^{2}_\tau}|Y^1_r-Y^2_r|\,dr.
\end{align}
We will show that
\begin{equation}
\label{eq3.2}
\lim_{k\rightarrow \infty}(Y^1_{\gamma^{1,k}_\tau\wedge\gamma^{2,k}_\tau}-Y^2_{\gamma^{1,k}_\tau\wedge\gamma^{2,k}_\tau})^+=0.
\end{equation}
The reasoning below is for fixed $\omega\in\Omega$. We consider several cases.

Case I: $\gamma^1_\tau=\tau\,\mbox{or}\, \gamma^2_\tau=\tau$.
Then $\gamma^{1,k}_\tau\wedge\gamma^{2,k}_\tau=\tau$. If $\tau<T$, then $Y^1_\tau=L^1_\tau$
or $Y^2_\tau=U^2_\tau$. In both cases (\ref{eq3.2}) is satisfied. If $\tau=T$, then the
limit in (\ref{eq3.2}) equals
$(\xi_1-\xi_2)^+$, so (\ref{eq3.2}) is satisfied  by the assumptions.

Case II: $\gamma^1_\tau>\tau\, \mbox{and}\, \gamma^2_\tau>\tau$.
We divide the proof into several sub-cases.

Case II(a): $\gamma^1_\tau<\gamma^2_\tau$.
First suppose that
there exists $k_0$ such that $\gamma^{1,k}_\tau\wedge\gamma^{2,k}_\tau=\gamma^1_\tau,\, k\ge k_0$. Then $\gamma^{1,k}_\tau=\gamma^1_\tau,\, k\ge k_0$. Hence $\omega\notin\Lambda^1_\tau$, which implies that $L^1_{\gamma^1_\tau}=U^1_{\gamma^1_\tau}$. Hence we get
easily (\ref{eq3.2}). Suppose now that
$\gamma^{1,k}_\tau\wedge\gamma^{2,k}_\tau<\gamma^1_\tau,\, k\ge 1$.
Then $\gamma^{1,k}_\tau<\gamma^1_\tau,\, k\ge 1$, which implies that $\omega\in\Lambda^1_\tau$.
Thus $L^1_{\gamma^1_\tau-}=U^1_{\gamma^1_\tau-}$. Therefore
\[
(Y^1_{\gamma^{1,k}_\tau\wedge\gamma^{2,k}_\tau}-Y^2_{\gamma^{1,k}_\tau\wedge\gamma^{2,k}_\tau})^+
\rightarrow (Y^1_{\gamma^1_\tau-}-Y^2_{\gamma^1_\tau-})^+=(L^1_{\gamma^1_\tau-}-Y^2_{\gamma^1_\tau-})^+=0.
\]

Case II(b): $\gamma^1_\tau>\gamma^2_\tau.$ The proof is analogous to that in Case II(a).

Case II(c): $\gamma^1_\tau=\gamma^2_\tau<T$. First suppose that
$\gamma^{1,k}_\tau\wedge\gamma^{2,k}_\tau<\gamma^1_\tau=\gamma^2_\tau$, $k\ge 1.$
Then $\gamma^{1,k}_\tau<\gamma^1_\tau,\, k\ge 1$ or $\gamma^{2,k}_\tau<\gamma^1_\tau,\, k\ge 1$, which
implies that $\omega\in \Lambda^1_\tau\cup \Lambda^2_\tau$ or equivalently $L^1_{\gamma^1_\tau-}=U^1_{\gamma^1_\tau-}\,\vee\, L^2_{\gamma^2_\tau-}=U^2_{\gamma^2_\tau-}$. Therefore
\[
(Y^1_{\gamma^{1,k}_\tau\wedge\gamma^{2,k}_\tau}-Y^2_{\gamma^{1,k}_\tau\wedge\gamma^{2,k}_\tau})^+
\rightarrow (Y^1_{\gamma^1_\tau-}-Y^2_{\gamma^1_\tau-})^+=(Y^1_{\gamma^2_\tau-}-Y^2_{\gamma^2_\tau-})^+.
\]
If $\omega\in\Lambda^1_\tau$, then $(Y^1_{\gamma^1_\tau-}-Y^2_{\gamma^1_\tau-})^+=(L^1_{\gamma^1_\tau-}-Y^2_{\gamma^1_\tau-})^+=0$. If $\omega\in\Lambda^2_\tau$, then $(Y^1_{\gamma^2_\tau-}-Y^2_{\gamma^2_\tau-})^+=(Y^1_{\gamma^2_\tau-}-U^2_{\gamma^2_\tau-})^+=0$.

Case II(d): $ \gamma^1_\tau=\gamma^2_\tau=T$. If there exists $k_0\in\BN$ such that
$\gamma^{1,k}_\tau\wedge \gamma^{2,k}_\tau=T$, then $(Y^1_{\gamma^{1,k}_\tau\wedge\gamma^{2,k}_\tau}-Y^2_{\gamma^{1,k}_\tau\wedge\gamma^{2,k}_\tau})^+=(\xi_1-\xi_2)^+=0,\, k\ge k_0$. If $\gamma^{1,k}_\tau\wedge \gamma^{2,k}_\tau<T,\, k\ge 1$, then $\omega\in \Lambda^1_\tau\cup \Lambda^2_\tau$ or equivalently $L^1_{T-}=U^1_{T-}\,\vee\, L^2_{T-}=U^2_{T-}$. If $\omega\in\Lambda^1_\tau$, then $(Y^1_{T-}-Y^2_{T-})^+=(L^1_{T-}-Y^2_{T-})^+=0$. If $\omega\in\Lambda^2_\tau$, then $(Y^1_{T-}-Y^2_{T-})^+=(Y^1_{T-}-U^2_{T-})^+=0$.
We have showed that (\ref{eq3.2}) is satisfied. Since $Y^1,Y^2$ are of class (D), it follows from (\ref{eq3.2}) that
\begin{equation}
\label{eq3.starl}
\lim_{k\rightarrow \infty} E(Y^1_{\gamma^{1,k}_\tau\wedge\gamma^{2,k}_\tau}-Y^2_{\gamma^{1,k}_\tau\wedge\gamma^{2,k}_\tau})^+=0.
\end{equation}
By this and (\ref{eq3.1}),
\begin{align}
\label{eq3.3}
E(Y^1_\tau-Y^2_\tau)^+\le
\mu^+E\int_\tau^{\gamma^{1}_\tau\wedge\gamma^{2}_\tau}(Y^1_r-Y^2_r)^+\,dr.
\end{align}
Let $a\ge 0$ be such that $T\le a$. Put $Y^1_t=\xi_1,\, Y^2_t=\xi_2,\, t\ge T$.
Then from (\ref{eq3.3}) we conclude that for every stopping time $\tau\le a$,
\begin{align*}
E(Y^1_\tau-Y^2_\tau)^+=E(Y^1_{\tau\wedge T}-Y^2_{\tau\wedge T})^+
&\le \mu^+E\int_{\tau\wedge T}^{a}(Y^1_r-Y^2_r)^+\,dr
\\
&=  \mu^+E\int_{\tau}^{a}(Y^1_r-Y^2_r)^+\,dr.
\end{align*}
Applying Gronwall's lemma yields $Y^1_t\le Y^2_t,\, t\in [0,T]$.
\end{proof}

\begin{wn}
Assume that \mbox{\rm(H2)} is satisfied. Then there exists at most one solution of {\rm RBSDE}$^T(\xi,f,L,U)$.
\end{wn}

\begin{tw}
\label{tw4.1}
Assume that \mbox{\rm(H1)--(H4)} are satisfied.
\begin{enumerate}
\item[\rm(i)] There exists a unique solution $(Y,\Gamma)$ of {\rm RBSDE}$^T(\xi,f,L,U)$.

\item[\rm(ii)] Let $\{\xi_n\}$ be a sequence of integrable $\FF_T$-measurable random variables
such that $\xi_n\nearrow \xi$, and let
\[
f_n(t,y)=f(t,y)+n(y-L_t)^-.
\]
Then  for each $n\in\BN$ there exists a unique solution $(Y^n,M^n,A^n)$ of the equation $\overline{R}BSDE^T(\xi_n,f_n,U)$, and moreover, $Y^n_t\nearrow Y_t,\, \Gamma^n_t\rightarrow \Gamma_t$, $t\in [0,T]$, where $\Gamma^n_t=\int_0^tn(Y^n_r-L_r)^-\,dr-A^n_t-M^n_t$, $t\in [0,T]$.
\end{enumerate}
\end{tw}
\begin{proof}
By  Proposition \ref{tw2.0}, for every $n\ge 1$ there exists a unique solution $(Y^n,M^n,A^n)$ of $\overline{{\rm{R}}}$BSDE$^T(\xi_n,f_n,U)$. By \cite[Proposition 2.1]{K:SPA2}, $Y^n\le Y^{n+1}$. Set
\[
Y_t=\lim_{n\rightarrow \infty}Y^n_t,\quad t\in [0,T].
\]
By \cite[Proposition 2.1]{K:SPA2} $Y^n\le \bar Y^n$, where
$(\bar Y^n,\bar M^n)$ is the solution of the problem BSDE$^T(\xi_n,f_n)$. By Proposition \ref{tw2.0}, $\bar Y^n\nearrow \bar Y$,
where $(\bar Y,\bar M,\bar K)$ is the solution of ${\rm{\underline{R}}}$BSDE$^T(\xi,f,L)$. Hence  $Y^1\le Y^n\le \bar Y$, $n\ge 1$, so $Y$ is of class (D).
By Proposition \ref{tw2.0}, for all $\varepsilon >0$ and $n\ge 1$ there exists a unique solution $(Y^{n,\varepsilon},M^{n,\varepsilon},A^{n,\varepsilon})$ of $\overline{{\rm{R}}}$BSDE$^T(\xi_n,f_{n,\varepsilon},U)$ with
\[
f_{n,\varepsilon}(t,y)=f(t,y)+n(y-L_t^\varepsilon)^-,\quad L^\varepsilon=L-\varepsilon.
\]
By \cite[Proposition 2.1]{K:SPA2}, $Y^{n,\varepsilon}\le Y^n$, while by Proposition \ref{tw2.1} and Remark \ref{uw2.1},
$Y^{n,\varepsilon}_t\nearrow Y^\varepsilon_t,\, t\in [0,T]$, where the triple $(Y^\varepsilon,M^\varepsilon,R^\varepsilon)$
is the unique solution to the problem RBSDE$^T(\xi,f,L^\varepsilon,U)$. Therefore $L^\varepsilon\le Y$, and since $\varepsilon>0$
was arbitrary, $L\le Y$. Of course $Y\le U$. Now we will show that $Y$ is c\`adl\`ag.
Let $\tau\in\TT$.
Applying Proposition \ref{tw2.1} (see also Remark \ref{uw2.1}, Remark \ref{uw.red}) on $[\tau,\gamma^k_\tau]$ (see (\ref{eq2.1}), (\ref{eq2.3cdf}) for the definition of $\gamma_\tau, \gamma^k_\tau$) with $\hat\xi^n= Y^n_{\gamma^k_\tau}$, we see that $Y$ is c\`adl\`ag on $[\tau,\gamma_\tau\}$. If $\tau<\gamma_\tau$, then we get that $Y$ is right-continuous in $\tau$, and
if $\tau=\gamma_\tau$, then $L_\tau=U_\tau=Y_\tau$, so $Y$ is right-continuous in $\tau$ by the right-continuity
of $L,U$ and the fact that $L\le Y\le U$. Hence, by \cite[IV.T28]{Dellacherie}, $Y$ is right-continuous on $[0,T]$.
Now let $\{\tau_m\}\subset \TT$ be an increasing sequence and $\tau:=\sup_{m\ge 1}\tau_m$.
It is clear that on the set $\{\omega\in\Omega;\,\tau_m(\omega)=\tau(\omega),\, m\ge m_\omega\}\cup\{L_{\tau-}=U_{\tau-}\}$ the limit $\lim_{m\rightarrow \infty}Y_{\tau_m}$ exists. Now we will show that this limit exists on the set
\[
A=\{\tau_m<\tau,\,m\ge 1\}\cap\{L_{\tau-}<U_{\tau-}\}.
\]
Applying Proposition \ref{tw2.1} (see also Remark \ref{uw2.1}, Remark \ref{uw.red}) on the interval
$[\tau_{m},\gamma^k_{\tau_m}]$  with $\hat\xi^n=Y^n_{\gamma^k_{\tau_m}}$ for every $k\ge 1$ we see that $Y$ is c\`adl\`ag on $[\tau_m,\gamma_{\tau_m}\}$. Since $A\subset \{L_{\tau-}<U_{\tau-}\}$, for every $\omega\in A$ there exists $m_\omega$ such that
\[
[\tau_{m_\omega}(\omega),\tau(\omega)]\subset [\tau_{m_\omega}(\omega),\gamma_{\tau_{m_\omega}}(\omega)\}.
\]
Therefore $\lim_{m\rightarrow \infty}Y_{\tau_m}$ exists on $A$. Summing up, we have that $\lim_{m\rightarrow \infty}Y_{\tau_m}$ exists a.s., so again by  \cite[IV.T28]{Dellacherie}, $Y$  has left limits on $[0,T]$.
Set
\[
\Gamma^n_t=\int_0^tn(Y^n_r-L_r)^-\,dr-A^n_t-M^n_t,\quad t\in [0,T].
\]
It is clear that
\[
Y^n_t=\xi_n+\int_t^T f(r,Y^n_r)\,dr+\int_t^T\,d\Gamma^n_r,\quad t\in [0,T].
\]
By (H2) and (H4) we may pass to the limit in the above equation. We then get 
condition (c) of the definition of RBSDE$^T(\xi,f,L,U)$
with $\Gamma_t= -Y_t+Y_0-\int_0^t f(r,Y_r)\,dr$. Let $\tau\in\TT$. Applying Proposition \ref{tw2.1} (see also Remark \ref{uw.red}) on $[\tau,\gamma_{\tau}^k],\, k\ge 1$ with $\hat\xi^n=Y^n_{\gamma_\tau^k}$ (see also Remark \ref{uw2.1}) we see that $\Gamma$ is a special semimartingale on $[\tau,\gamma_\tau\}$ and
\[
\int_{[\tau,\gamma_\tau\}}(Y_{r-}-L_{r-})\,d\Gamma^{v,+}_r(\tau)=\int_{[\tau,\gamma_\tau\}}(U_{r-}-Y_{r-})\,d\Gamma^{v,-}_r(\tau)=0,
\]
which completes the proof.
\end{proof}

\begin{wn}
\label{wn.pna}
Assume that \mbox{\rm (H1)--(H4)} are satisfied. Let $(Y,\Gamma)$ be a solution of {\rm RBSDE}$^T(\xi,f,L,U)$ and $(Y^n,M^n)$ be a solution of {\rm BSDE}$^T(\xi,f_n)$ with
\[
f_n(t,y)=f(t,y)+n(y-L_t)^--n(y-U_t)^+.
\]
Then $Y^n_t\rightarrow Y_t$, $t\in [0,T]$.
\end{wn}
\begin{proof}
By Theorem \ref{tw4.1}, $\overline{Y}^n\nearrow Y$, where $(\overline Y^n,\overline M^n,\overline A^n)$ is a solution of the equation $\overline{{\rm{R}}}$BSDE$^T(\xi,\overline f_n,U)$ with $\overline f_n(t,y)=f(t,y)+n(y-L_t)^-$. In  much the same manner one can show that $\underline{Y}^n\searrow Y$, where $(\underline Y^n,\underline M^n,\underline K^n)$ is a solution of $\underline{{\rm{R}}}$BSDE$^T(\xi,\underline f_n,U)$ with $\underline f_n(t,y)=f(t,y)-n(y-U_t)^+$. By \cite[Proposition 2.1]{K:SPA2}, $\overline Y^n\le Y^n\le \underline Y^n$, which implies the desired result.
\end{proof}

\begin{wn}
\label{wn4.chv}
Let $\alpha\in\BR$ and $(Y,\Gamma)$ be a solution of the problem {\rm RBSDE}$^T(\xi,f,L,U)$. Then $(Y^\alpha,\Gamma^\alpha)$
is a solution of {\rm RBSDE}$^T(\xi^\alpha,f^\alpha,L^\alpha,U^\alpha)$ with
\[
\xi^\alpha=e^{\alpha T}\xi,\quad f^\alpha(t,y)=e^{\alpha t}f(t,e^{-\alpha t}y)-\alpha y,\quad L^\alpha_t=e^{\alpha t}L_t,\quad U^\alpha_t=e^{\alpha t}U_t,
\]
where
\[
Y^\alpha_t=e^{\alpha t}Y_t,\qquad \Gamma^\alpha_t=e^{\alpha t}\Gamma_t-\int_0^t\alpha e^{\alpha r}\Gamma_r\,dr.
\]
\end{wn}
\begin{proof}
We first assume that $E\int_0^T|f(r,Y_r)|\,dr<\infty$. By Theorem \ref{tw4.1}, $Y^n_t\nearrow Y_t$, $t\in [0,T]$,
and $\Gamma^n_t\rightarrow \Gamma_t$, $t\in [0,T]$, where $(Y^n,M^n,A^n)$ is a solution of  $\overline{{\rm{R}}}$BSDE$^T(\xi_n,f_n,U)$ with $f_n(t,y)=f(t,Y_t)+n(y-L_t)^-$ and
\[
\Gamma^n_t=\int_0^tn(Y^n_r-L_r)^-\,dr-A^n_t-M^n_t,\quad t\in [0,T].
\]
It is clear that
\[
Y^n_t=\xi+\int_t^T f(r,Y_r)\,dr+\int_t^T\,d\Gamma^n_r,\quad t\in [0,T].
\]
Integrating by parts,  we obtain
\begin{equation}
\label{eq.alp}
e^{\alpha t}Y^{n}_t=\xi^\alpha+\int_t^T e^{\alpha r}f(r,Y_r)\,dr-\int_t^T\alpha e^{\alpha r}Y^n_r\,dr+\int_t^T\,d\Gamma^{n,\alpha}_r,\quad t\in [0,T],
\end{equation}
with
\[
\Gamma^{n,\alpha}_t=e^{\alpha t}\Gamma^n_t-\int_0^t\alpha e^{\alpha r}\Gamma^n_r\,dr.
\]
Therefore letting $n\rightarrow\infty$ in (\ref{eq.alp}) we get
\[
Y^{\alpha}_t=\xi^\alpha+\int_t^T f^\alpha(r,Y^\alpha_r)\,dr+\int_t^T\,d\Gamma^{\alpha}_r,\quad t\in [0,T].
\]
It is clear that $Y^\alpha$ is of class (D)  and $L^\alpha\le Y^\alpha\le U^\alpha$. What is left is to show that condition (d) of Definition \ref{df.main} is satisfied. However, this condition easily follows from the fact that on the interval $[\tau,\gamma_\tau\}$ we have
\[
\Gamma^{\alpha,v}_t-\Gamma^{\alpha,v}_\tau=\int_\tau^te^{\alpha r}\,d\Gamma^v_r.
\]
Let $\{\tau_k\}$ be a chain
on $[0,T]$ such that $E\int_0^{\tau_k}|f(r,Y_r)|\,dr<\infty,\, k\ge 1$. By what has 
already been proved the pair $(Y^\alpha,\Gamma^\alpha)$ is a solution to RBSDE$^{\tau_k}(Y^\alpha_{\tau_k},f^\alpha,L^\alpha,U^\alpha)$.   Since $\{\tau_k\}$ is a chain, we get the  result.
\end{proof}

The following theorem shows that the solutions of reflected  BSDEs are stable with respect to  the norm
$\|\cdot\|_{1;T}$ defined by (\ref{eq2.dn1}).
\begin{tw}
\label{tw.stb}
Let  $(Y^i,\Gamma^i)$ be a solution of  {\rm RBSDE}$^T(\xi^i,f^i,L^i, U^i)$, $i=1,2$,  and $f^1$  satisfy \mbox{\rm(H2)}. Then for all $\tau\in\TT$ and $\varepsilon>0$,
\begin{align*}
(Y^1_\tau-Y^2_\tau)^+ &\le E\Big(e^{(T-\tau)\mu^+}(\xi^1-\xi^2)^+\\&\quad+\int_\tau^{\hat\beta_\tau}e^{(r-\tau)\mu^+}(f_1(r,Y^2_r)-f_2(r,Y^2_r))^+\,dr\\&\quad
+e^{(\hat\beta_\tau-\tau)\mu^+ }\mathbf{1}_{\hat\beta_\tau<T}[(L^1_{\hat\beta_\tau}-L^2_{\hat\beta_\tau})^++
(U^1_{\hat\beta_\tau}-U^2_{\hat\beta_\tau})^+]|{\FF_\tau}\Big)+\varepsilon,
\end{align*}
where $\hat\beta_\tau= \beta^1_\tau\wedge\beta^2_\tau$ and
\[
\beta^1_\tau=\inf\{t\ge\tau: Y^1_t\le L^1_t+\varepsilon\}\wedge T,\quad \beta^2_\tau=\inf\{t\ge\tau: Y^2_t\ge U^2_t-\varepsilon\}\wedge T.
\]
\end{tw}
\begin{proof}
By Corollary \ref{wn4.chv}, we may assume that $\mu^+=0$. Let $\tau\in\TT$ and $\gamma^i_\tau,\, \{\gamma^{i,k}_\tau\}$ be defined as in (\ref{eq2.1}), (\ref{eq2.3cdf}) but with $L,U$ replaced by  $L^i,U^i$.
We put $\hat\gamma^k_\tau=\gamma^{1,k}_\tau\wedge\gamma^{2,k}_\tau$, $\hat\gamma_\tau=\gamma^1_\tau\wedge\gamma^2_\tau$ and $\sigma^k_\tau=\hat\beta_\tau\wedge\hat\gamma^k_\tau$. Observe that
$\hat\beta_\tau\le \hat\gamma_\tau$. By the minimality condition (d) in Definition \ref{df.main}
and the definition of $\hat\beta_\tau$, we have
\[
\int_\tau^{\sigma_\tau^k}\, d\Gamma^{1,v,+}(\tau)_r+\int_\tau^{\sigma_\tau^k}\,d\Gamma^{2,v,-}(\tau)_r=0.
\]
Therefore  applying the Tanaka-Meyer formula on $[\tau,\sigma_\tau^k]$ and using (H2) we get
\begin{equation}
\label{eq4.8}
(Y^1_{\tau}-Y^2_{\tau})^+\le  E\Big((Y^1_{\sigma^k_\tau}-Y^2_{\sigma^k_\tau})^+
+\int_\tau^{\sigma^k_\tau}(f_1(r,Y^2_r)-f_2(r,Y^2_r))^+\,dr|{\FF_\tau}\Big).
\end{equation}
The following calculations are made for fixed $\omega\in\Omega$.
We consider two cases.

Case I: $\hat\gamma_\tau=\tau$.
If $ \tau<T$, then $L^1_{\sigma^k_\tau}=U^1_{\sigma^k_\tau}=Y^1_{\sigma^k_\tau}$, $k\ge 1$ or
$L^2_{\sigma^k_\tau}=U^2_{\sigma^k_\tau}=Y^2_{\sigma^k_\tau}$, $k\ge 1$. In both cases  we have
\begin{align}
\label{eq4.9}
(Y^1_{\sigma^k_\tau}-Y^2_{\sigma^k_\tau})^+&\le \max\{(L^1_{\sigma^k_\tau}-L^2_{\sigma^k_\tau})^+, (U^1_{\sigma^k_\tau}-U^2_{\sigma^k_\tau})^+\}\nonumber\\
&=\max\{(L^1_{\hat\beta_\tau}-L^2_{\hat\beta_\tau})^+, (U^1_{\hat\beta_\tau}-U^2_{\hat\beta_\tau})^+\}.
\end{align}
If $\tau=T$, then $(Y^1_{\sigma^k_\tau}-Y^2_{\sigma^k_\tau})^+=(\xi^1-\xi^2)^+$.

Case II: $\hat\gamma_\tau>\tau$. We consider the following three sub-cases.

Case II(a): $\hat\beta_\tau\in [\tau,\hat\gamma_\tau)$. Then $\hat\beta_\tau<\hat\gamma^k_\tau$, $k\ge k_0$. Moreover, $Y^1_{\sigma^k_\tau}\le L^1_{\sigma^k_\tau}+\varepsilon$ or $Y^2_{\sigma^k_\tau}\ge U^2_{\sigma^k_\tau}-\varepsilon$, $k\ge k_0$. Therefore
\begin{align*}
(Y^1_{\sigma^k_\tau}-Y^2_{\sigma^k_\tau})^+&\le \max\{(L^1_{\sigma^k_\tau}-L^2_{\sigma^k_\tau})^+,(U^1_{\sigma^k_\tau}-U^2_{\sigma^k_\tau})^+\}
+\varepsilon\\&=\max\{(L^1_{\hat\beta_\tau}
-L^2_{\hat\beta_\tau})^+,(U^1_{\hat\beta_\tau}-U^2_{\hat\beta_\tau})^+\}+\varepsilon,\quad k\ge k_0.
\end{align*}

Case II(b): $\hat\gamma_\tau=\hat\beta_\tau<T$. Then $\omega\notin \Lambda^1_\tau\cup\Lambda^2_\tau$. Hence $L^1_{\sigma^k_\tau}=U^1_{\sigma^k_\tau}$ or $L^2_{\sigma^k_\tau}=U^2_{\sigma^k_\tau}$, $k\ge k_0$. In both cases (\ref{eq4.9}) is satisfied.

Case II(c): $\hat\gamma_\tau=\hat\beta_\tau=T$. Then $\omega\notin \Lambda^1_\tau\cup\Lambda^2_\tau$, so $(Y^1_{\sigma^k_\tau}-Y^2_{\sigma^k_\tau})^+=(\xi_1-\xi_2)^+$, $k\ge k_0$.
Combining Case I with Case II and (\ref{eq4.8}), we get the desired result.

\end{proof}
\medskip

\begin{wn}
Let  $(Y^i,\Gamma^i)$ be a solution of  {\rm RBSDE}$^T(\xi^i,f^i,L^i, U^i)$, $i=1,2$,  and $f^1$  satisfy \mbox{\rm(H2)} with $\mu\le0$, then (\ref{eq1.9}) holds.

\end{wn}

\section{Reflected BSDEs  with  arbitrary terminal time}
\label{sec4}

In this section, we prove an existence and uniqueness result for reflected BSDEs without any restriction on the terminal time $T$.
Moreover, we show that if the  processes $L$ and $-U$ are subregular, i.e. (\ref{eqi.2.3}) holds, then the convergence in the penalization scheme is uniform on compacts.

We  modify the definition of the set $\Lambda_\tau$ introduced in Section \ref{sec3}. Now, we set
\[
\Lambda_\tau=\{L_{\gamma_\tau-}=U_{\gamma_\tau-}\}\cap\{\tau<\gamma_\tau<\infty\}.
\]
 We also put $[\alpha(\omega),\beta(\omega)\}=[\alpha(\omega),\infty)$
if $\beta(\omega)=\infty$.

The main difference  between reflected BSDEs with bounded and unbounded terminal times lies in the definition of a solution, especially in
formulation of terminal condition which in the general case is of the following form
\begin{equation}
\label{eq5.1tutu}
Y_{T\wedge a}\rightarrow \xi,\quad a\rightarrow \infty.
\end{equation}
Moreover, in case of unbounded terminal times, we assume additionally that $\mu\le0$ in hypothesis  (H2).
One another difficulty which appears in the case of unbounded terminal time  concerns the integrability of $f$.
For bounded terminal time, $f_{n}(\cdot,S)$ (see Theorem \ref{tw4.1} for the definition of $f_n$) is integrable if and only if $f(\cdot,S)$ is integrable
for $S$ appearing in (H1), because $L^+,S$ are of class (D). This is no longer true for unbounded terminal time. This forces some additional assumptions when considering the  penalization scheme.

\begin{df}
\label{df.mainun}
We say that a pair $(Y,\Gamma)$ of $\BF$-adapted c\`adl\`ag processes  is a solution of the
reflected backward stochastic differential equation on the interval $[0,T]$ with terminal condition $\xi$, generator $f$, lower barrier $L$ and upper barrier $U$ (RBSDE$^T(\xi,f,L,U)$ for short) if for every $a\ge 0$  it is a solution of RBSDE$^{T\wedge a}(Y_{T\wedge a},f,L,U)$  and (\ref{eq5.1tutu}) is satisfied.
\end{df}

\begin{uw}
\label{uw.gen1}
Let (H2) be satisfied with $\mu\le 0$. Then  Theorem \ref{tw.unq} and Theorem \ref{tw.stb} hold true for
unbounded $T$. The proofs of these results  run, without any changes, as the proofs of Theorems  \ref{tw.unq} and  \ref{tw.stb} (the proof of Theorem \ref{tw.unq} is even simpler since the right-hand side of (\ref{eq3.3}) equals zero).
\end{uw}

\begin{uw}
\label{uw.ps}
A brief inspection of the proofs reveals that
all the results of  Section \ref{sec3} concerning the convergence
of the penalization schemes, i.e.  schemes including the term $n(y-L_t)^+$ or $n(y-U_t)^-$ (see Theorem \ref{tw4.1} (ii), Corollary \ref{wn.pna}), remain valid if we replace  the constants $n$ by any  positive bounded  $\BF$-progressively measurable processes $N^n$ such that $N^n_t\nearrow \infty$ a.s. as $n\rightarrow \infty$ for every
$t\in [0,T]$.
\end{uw}

From now on, $\eta$ is a  strictly positive bounded $\BF$-progressively measurable process such that $E\int_0^T\eta_t(S_t-L_t)^-\,dt+E\int_0^T\eta_t(S_t-U_t)^+\,dt<\infty$, where $S$ is a process from condition (H1). Such a process $\eta$ always exists. For instance, the process defined as
\[
\eta_t=\frac2\pi \frac{1}{1+t^2}\,,\quad t\ge 0,
\]
has the desired property because
\[
E\int_0^T\eta_t(S_t-L_t)^-\,dt+E\int_0^T\eta_t(S_t-U_t)^+\,dt\le \|S\|_{1;T}+\|L^+\|_{1;T}+\|U^-\|_{1;T}\,.
\]

\begin{tw}
\label{tw5.2}
Assume that \mbox{\rm(H1)--(H4)} are satisfied  with $\mu\le 0$.
\begin{enumerate}
\item[\rm(i)] There exists a unique solution $(Y,\Gamma)$ of RBSDE$^T(\xi,f,L,U)$.

\item[\rm(ii)] Let $\{\xi_n\}$ be an increasing sequence of integrable $\FF_T$-measurable random variables such that $\xi_n\nearrow \xi$, and let
 \[
f_n(t,y)=f(t,y)+n\eta_t(y-L_t)^-.
\]
Then for each $n\in\BN$ there exists a unique solution $(Y^n,M^n,A^n)$ of the equation $\overline{R}BSDE^T(\xi_n,f_n,U)$.
Moreover, $Y^n_t\nearrow Y_t$ and $ \Gamma^n_t\rightarrow \Gamma_t$, $t\in [0,T\wedge a]$, where $\Gamma^n_t=\int_0^tn\eta_t(Y^n_r-L_r)^-\,dr-A^n_t-M^n_t,\, t\in [0,T\wedge a],\, a\ge 0$.
\end{enumerate}
\end{tw}
\begin{proof}
By  Proposition \ref{tw5.1}, for each $n\in\BN$ there exists a unique solution $(Y^n,M^n,A^n)$
of $\overline{R}BSDE^T(\xi_n,f_n,U)$. By Theorem \ref{tw.unq},
$Y^n\le Y^{n+1}$. Set
\[
Y_t=\lim_{n\rightarrow \infty} Y^n_t,\quad t\in [0,T].
\]
Observe that $Y^n\le \bar Y^n$, where
$(\bar Y^n,\bar M^n)$ is the solution of BSDE$^T(\xi_n,f_n)$. By Proposition \ref{tw5.1}, $\bar Y^n\nearrow \bar Y$,
where $(\bar Y,\bar M,\bar K)$ is a solution of ${\rm{\underline{R}}}$BSDE$^T(\xi,f,L)$. Hence  $Y^1\le Y^n\le \bar Y$, $n\ge 1$, so $Y$ is of class (D). From  Proposition \ref{tw5.1} and Remark \ref{uw.ps} applied on the interval $[0,T\wedge a]$ it follows that $Y^n\nearrow Y^a$, where $(Y^a,\Gamma^a)$ is a solution of RBSDE$^{T\wedge a}(Y_{T\wedge a},f,L,U)$. Set $\Gamma_t=\Gamma^a_t$, $t\in [0,T\wedge a]$. It is clear that $\Gamma$ is well defined. What is left is to show that (\ref{eq5.1tutu}) is satisfied.  But this is a consequence of the inequality  $Y^1\le Y\le \bar Y$.
\end{proof}
\medskip

The following theorem says that under  Mokobodzki's condition a solution in the sense of Definition \ref{df.mainun} becomes semimartingale  solution.

\begin{tw}
\label{tw4.ident}
Assume that $(Y,\Gamma)$ is a solution to {\rm RBSDE}$^T(\xi,f,L,U)$. If there exists a special semimartingale between the barriers $L,U$, then $Y,\Gamma$ are special semimartingales and the triple $(Y,\Gamma^v,\Gamma^m)$
is a semimartingale  solution of {\rm RBSDE}$^T(\xi,f,L,U)$, i.e. in the sense of Definition \ref{df2.dfsd}, where $\Gamma^v$ (resp. $\Gamma^m$) is a predictable finite variation part (resp. local martingale part) from the Doob-Meyer decomposition of the special semimartingale $\Gamma$.
\end{tw}
\begin{proof}
Let $\{\theta_k\}$ be a chain on $[0,T]$ such that $E\int_0^{\theta_k}|f(r,Y_r)|\,dr<\infty$ for every $k\ge 1$, and let
$\tau_k=\theta_k\wedge k$. Write $f_Y(t)= f(t,Y_t)$. It is clear that $(Y,\Gamma)$ is a solution of RBSDE$^{\tau_k}(Y_{\tau_k},f_Y,L,U)$ for every $k\ge 1$. By  Theorem \ref{tw5.2}, $Y_t^{k,n}\nearrow Y_t,\, t\in [0,\tau_k]$, where $(Y^{k,n},A^{k,n},M^{k,n})$ is a solution of ${\rm{\overline{R}BSDE}}^{\tau_k}(Y_{\tau_k},f^n_Y,U)$ with
\[
f^n_Y(t,y)=f_Y(t)+n\eta_t(y-L_t)^-.
\]
On the other hand, by Proposition \ref{tw2.1}, $Y_t^{k,n}\nearrow \tilde Y^k_t,\, t\in [0,\tau_k]$, where the triple $(\tilde Y^k, \tilde R^k, \tilde M^k)$ is a semimartingale  solution of RBSDE$^T(Y_{\tau_k},f_Y,L,U)$. By Theorem \ref{tw.unq}, $Y=\tilde Y^k$ on $[0,\tau_k],\, k\ge 1$. From this the result follows.
\end{proof}

\begin{uw}
\label{uw.dps}
Under the assumptions of Theorem \ref{tw5.2}, $Y^n\rightarrow Y$, where $(Y^n,M^n)$ is a solution of
{\rm BSDE}$^T(\xi,f_n)$ and
\[
f_n(t,y)=f(t,y)+n\eta_t(y-L_t)^--n\eta_t(y-U_t)^+.
\]
To see this,  we denote by $(\overline Y^n,\overline M^n,\overline A^n)$  a solution of ${\rm\overline{R}BSDE}^T(\xi,\overline f_n,U)$ with
\[
\overline f_n(t,y)=f(t,y)+n\eta_t(y-L_t)^-,
\]
 and by $(\underline Y^n,\underline M^n,\underline K^n)$ a solution of  ${\rm\underline{R}BSDE}^T(\xi,\underline f_n,L)$ with
\[
\underline f_n(t,y)=f(t,y)-n\eta_t(y-U_t)^+.
\]
By Theorem \ref{tw5.2}, $\overline Y^n\nearrow Y$ and $\underline Y^n\searrow Y$, whereas by Theorem \ref{tw.unq},
$\overline Y^n\le Y^n\le \underline Y^n$, from which the desired result follows.
\end{uw}

\begin{tw}
\label{tw5.prj}
Assume that \mbox{\rm(H1)--(H4}) with $\mu\le 0$ are satisfied and  $^pL\ge L_-$\,, $^pU\le U_-$\,. Then
\[
\bar Y_n\rightarrow Y\quad \mbox{in ucp},
\]
where $(\bar Y^n,\bar M^n,\bar A^n)$, $(Y,\Gamma)$ are processes defined in Theorem \ref{tw5.2}.
\end{tw}
\begin{proof}
By Theorem \ref{tw5.2},  $\bar Y^n\nearrow Y$, so $(\bar Y^n-L)^-\searrow 0$, and hence $^p(\bar Y^n-L)^-\searrow 0$.
By the assumption on $U$ and \cite[Proposition 4.3]{K:SPA2}, $\bar A^n$ is continuous, so $^p\bar Y^n=Y^n_-$. Therefore by the assumption on $L$,
\[
^p(\bar Y^n-L)^-=(\bar Y^n_--\,^pL)^-\ge (\bar Y^n_--L_-)^-.
\]
Consequently, $(\bar Y^n-L)^-\searrow 0$ and  $(\bar Y^n_--L_-)^-\searrow 0$, which by Dini's theorem implies that $(\bar Y^n-L)^-\rightarrow 0$ in ucp.  Since $0\le (\bar Y^n-L)^-\le |\bar Y^1|+L^+$ and $\bar Y^1, L^+$ are of class (D), it follows that for every $a\ge0$, $\|(\bar Y^n-L)^-\|_{1;T\wedge a}\rightarrow 0$ as $n\rightarrow\infty$. Observe that
the triple $(\bar Y^n,\bar M^n,\bar R^n)$ is a solution of RBSDE$^T(\xi_n,f,L_n,U)$ with
\[
L_n=L-(\bar Y^n-L)^-,\quad \bar R^n=n(\bar Y^n-L)^--\bar A^n.
\]
By Theorem \ref{tw.stb},  $\|\bar Y^n-Y\|_{1;T\wedge a}\le E|\bar Y^n_{T\wedge a}-Y_{T\wedge a}|+\|(\bar Y^n-L)^-\|_{1;T\wedge a}$. Combining the above arguments, we easily obtain the desired result.
\end{proof}

\begin{wn}
\label{wn5.prj}
Under the assumptions of Theorem \ref{tw5.prj},
\[
Y_n\rightarrow Y\quad \mbox{in ucp},
\]
where $(Y^n,M^n)$ is defined in Remark \ref{uw.dps}.
\end{wn}
\begin{proof}
See the reasoning in Remark \ref{uw.dps}.
\end{proof}

\section{Dynkin games, RBSDEs and nonlinear $f$-expectation}
\label{sec5}

In this section, we  assume additionally that $L,U$ are of class (D) (not merely $L^+$, $U^-$).
As in Section \ref{sec4} terminal time $T$ is an arbitrary stopping time.
\subsection{Dynkin games and RBSDEs}
\begin{tw}
\label{tw.obs}
Let  $(Y,\Gamma)$ be a solution of {\rm RBSDE}$^T(\xi,f,L,U)$. Assume additionally that  $E\int_0^T|f(r,Y_r)|\,dr<\infty$. Then for every $\alpha\in\TT$,
\begin{align}
\label{eq4.5}
Y_\alpha= \esssup_{\sigma\ge\alpha}\essinf_{\tau\ge\alpha}J_\alpha(\tau,\sigma)
=\essinf_{\tau\ge\alpha}\esssup_{\sigma\ge\alpha}J_\alpha(\tau,\sigma),
\end{align}
where
\begin{align}
\label{eq.vf}
J_\alpha(\tau,\sigma)=E\Big(\int_\alpha^{\tau\wedge\sigma}f(r,Y_r)\,dr
+L_\sigma\mathbf{1}_{\sigma<\tau}
+U_\tau\mathbf{1}_{\tau\le\sigma<T}+\xi\mathbf{1}_{\sigma=\tau=T}|\FF_\alpha\Big).
\end{align}
Moreover, for all $\sigma,\tau\in \TT_\alpha$,
\begin{equation}
\label{eq5.aagg}
J_\alpha(\tau_\varepsilon,\sigma)-\varepsilon\le Y_\alpha\le J_\alpha(\tau,\sigma_\varepsilon)+\varepsilon,
\end{equation}
where
\begin{equation}
\label{eq5.4}
\tau_\varepsilon=\inf\{t\ge\alpha: Y_t\ge U_t-\varepsilon\}\wedge T,\qquad
\sigma_\varepsilon=\inf\{t\ge\alpha: Y_t\le L_t+\varepsilon\}\wedge T.
\end{equation}
\end{tw}
\begin{proof}
Let $\tau, \sigma\in\TT_\alpha$. It is clear that $\sigma_\varepsilon, \tau_\varepsilon\le \gamma_\alpha$. Let $\{\delta_n\}$
be a fundamental sequence for the local martingale $\Gamma^m(\alpha)$ on $[\alpha,\gamma_\alpha\}$, and let
\[
\theta_k=\tau_\varepsilon\wedge\sigma\wedge\gamma^k_\alpha.
\]
By the minimality condition on $U$ (see Definition \ref{df.main}(d)),
\begin{align}
\label{eq.wenyh}
Y_{\alpha}=Y_{\theta_k\wedge\delta_n}+\int_{\alpha}^{\theta_k\wedge\delta_n} f(r,Y_r)\,dr
+\int_{\alpha}^{\theta_k\wedge\delta_n}\,d\Gamma^{v,+}_r(\alpha)-\int_{\alpha}^{\theta_k\wedge\delta_n}\,dM_r.
\end{align}
Since $Y$ is of class (D), taking the conditional expectation with respect to $\FF_\alpha$ of both sides of the above equality and then letting $n\rightarrow \infty$, we get (observe that $(\theta_k\wedge\delta_n)(\omega)=\theta_k(\omega),\, n\ge n_0(\omega)$)
\begin{equation}
\label{eq5.ccaa}
Y_\alpha=E\Big(Y_{\theta_k}+\int_\alpha^{\theta_k}f(r,Y_r)\,dr
+\int_\alpha^{\theta_k}d\Gamma^{v,+}_r(\alpha)|\FF_{\alpha}\Big).
\end{equation}
As $k\rightarrow\infty$, we have
\begin{equation}
\label{eq4.7}
Y_{\theta_k}\rightarrow Y_{\tau_\varepsilon\wedge\sigma}.
\end{equation}
To see this, let us consider two cases: (a) $\theta_k=\tau_\varepsilon\wedge\sigma$ for some  $k\ge k_0$ ($k_0$ depends on $\omega$), and (b) $\theta_k<\tau_\varepsilon\wedge\sigma$, $k\ge1$. Let us fix $\omega\in\Omega$. It is clear that (\ref{eq4.7}) is satisfied  in case (a).
In case (b), $\omega\notin \Lambda_\alpha$ for  otherwise we would have $\tau_\varepsilon <\gamma_\alpha$ (since $L_{\gamma_\alpha-}=U_{\gamma_\alpha-}$ if $\omega\in \Lambda_\alpha$),  which in turn implies (a) since $\gamma^k_\alpha\nearrow \gamma_\alpha$. Hence, in case (b), $\gamma^k_\alpha<\gamma_\alpha$, $k\ge 1$ and $\omega\notin \Lambda_\alpha$. This is possible only if $\gamma_\alpha=\infty$, so $\gamma_\alpha^k\rightarrow \infty$. Since $\theta_k=\gamma^k_\alpha$ in case (b), it follows that $\theta_k\rightarrow \infty$ and $\tau_\varepsilon\wedge\sigma=\infty$, which implies that $Y_{\theta_k} \rightarrow\xi=Y_{\tau_\varepsilon\wedge\sigma}$, i.e. (\ref{eq4.7}) is satisfied. Letting $k\rightarrow\infty$  in (\ref{eq5.ccaa}), using (\ref{eq4.7}) and the fact that $\theta_k\rightarrow \tau_\varepsilon\wedge\sigma$, we  get
\[
Y_\alpha\ge E\Big(Y_{\tau_\varepsilon\wedge \sigma}+\int_\alpha^{\tau_\varepsilon\wedge \sigma}f(r,Y_r)\,dr |\FF_{\alpha}\Big).
\]
By the definition of $\tau_\varepsilon$ and the fact that $Y\ge L$, we conclude from the above inequality that
\begin{align*}
Y_\alpha&\ge E\Big(Y_{\tau_\varepsilon}\mathbf{1}_{\tau_\varepsilon\le \sigma<T}+Y_{\sigma}\mathbf{1}_{\sigma<\tau_\varepsilon}+\xi\mathbf{1}_{\tau_\varepsilon=\sigma=T}+\int_\alpha^{\tau_\varepsilon\wedge \sigma}f(r,Y_r)\,dr |\FF_{\alpha}\Big)\\&\ge
E\Big((U_{\tau_\varepsilon}-\varepsilon)\mathbf{1}_{\tau_\varepsilon\le \sigma<T}+L_{\sigma}\mathbf{1}_{\sigma<\tau_\varepsilon}+\xi\mathbf{1}_{\tau_\varepsilon=\sigma=T}+\int_\alpha^{\tau_\varepsilon\wedge \sigma}f(r,Y_r)\,dr |\FF_{\alpha}\Big).
\end{align*}
Hence
\[
J_\alpha(\tau_\varepsilon,\sigma)-\varepsilon\le Y_\alpha.
\]
A similar argument applied to the  pair $\tau,\sigma_\varepsilon$ gives the second inequality in (\ref{eq5.aagg}). From (\ref{eq5.aagg}) we easily deduce (\ref{eq4.5}).
\end{proof}

\begin{wn}
\label{wn.viwn}
Assume that $Y$ is a progressively measurable process  such that we have $E\int_0^T|f(r,Y_r)|\,dr<\infty$ and (\ref{eq4.5}) holds for every $\alpha\in\TT$. Then $Y$ is a solution
of {\rm RBSDE}$^T(\xi,f,L,U)$.
\end{wn}
\begin{proof}
By Theorem \ref{tw5.2}, there exists a unique solution $\bar Y$ to RBSDE$^T(\xi,f_Y,L,U)$ with
$f_Y(t)=f(t,Y_t)$. By Theorem \ref{tw.obs}, for every $\alpha\in\TT$,
\begin{align*}
\bar Y_\alpha= \esssup_{\sigma\ge\alpha}\essinf_{\tau\ge\alpha}J_\alpha(\tau,\sigma)
=\essinf_{\tau\ge\alpha}\esssup_{\sigma\ge\alpha}J_\alpha(\tau,\sigma),
\end{align*}
where $J_\alpha(\tau,\sigma)$ is given by (\ref{eq.vf}). Thus $Y=\bar Y$, so $Y$ is a solution of the problem RBSDE$^T(\xi,f,L,U)$.
\end{proof}

\begin{tw}
\label{tw.sp}
Assume that
\begin{equation}
\label{eq5.prj}
^pL\ge L_-\,,\quad ^pU\le U_-\,.
\end{equation}
Let  $(Y,\Gamma)$ be a solution of {\rm RBSDE}$^T(\xi,f,L,U)$ such that $E\int_0^T|f(r,Y_r)|\,dr<\infty$. Then for every $\alpha\in\TT$,
\begin{equation}
\label{eq5.abc}
Y_\alpha=J_\alpha(\sigma^*_\alpha,\tau^*_\alpha),
\end{equation}
where $J_\alpha$ is given by \mbox{\rm(\ref{eq.vf})} and
\begin{equation}
\label{eq5.sp}
\sigma^*_\alpha=\inf\{t\ge\alpha:Y_t=L_t\}\wedge T,\qquad \tau^*_\alpha=\inf\{t\ge\alpha:Y_t=U_t\}\wedge T.
\end{equation}
\end{tw}
\begin{proof}
Step 1. We assume additionally that $Y$ (or, equivalently, $\Gamma$) is a special semimartingale. Under this additional condition we will show that
\begin{equation}
\label{eq4.vth}
\int_\alpha^{\sigma^*_\alpha}\,d\Gamma^{v,+}_r=\int_\alpha^{\tau^*_\alpha}\,d\Gamma^{v,-}_r=0.
\end{equation}
By Theorem \ref{tw4.ident}, the triple $(Y,\Gamma^v,\Gamma^m)$ is a semimartingale solution of the equation RBSDE$^T(\xi,f,L,U)$. Let $\{\tau_k\}$ be a fundamental sequence for the local martingale $\Gamma^m_\cdot-\Gamma^m_\alpha$ on $[[\alpha,T]]$. We set
\[
\theta_{k}= \tau^*_\alpha\wedge\sigma^*_\alpha\wedge \tau_k,
\]
and then
\[
A_k=\{(t,\omega)\in ]]\alpha,\theta_k]]: Y_{t-}(\omega)=L_{t-}(\omega),\, \Delta \Gamma^{v,+}_t(\alpha)(\omega)>0\},
\]
\[
B_k=\{(t,\omega)\in ]]\alpha,\theta_k]]: Y_{t-}(\omega)=U_{t-}(\omega),\, \Delta \Gamma^{v,-}_t(\alpha)(\omega)>0\}.
\]
We will show that $P(\Pi(A_k))=P(\Pi(B_k))=0$. Assume that $P(\Pi(A_k))>0$.  Since $A_k$ is predictable, by the Section Theorem, for every $\varepsilon>0$  there exists a
predictable stopping time $\tau$ (depending on $k,\varepsilon$) such that
\begin{equation}
\label{eq5.sap1}
[[\tau]]\subset A_k,\quad P(\Pi(A_k))\le P(\tau<\infty)+\varepsilon.
\end{equation}
Observe that on the set $\{\tau<\infty\}$ we have
\begin{equation}
\label{eq5.sap2}
Y_\tau-L_{\tau-}+\Delta \Gamma^{v,+}_\tau=\Delta \Gamma^m_\tau.
\end{equation}
Since $\tau$ is predictable and  $L_{\tau}\le Y_{\tau}$, we have $E\mathbf{1}_{\{\tau<\infty\}}(Y_{\tau}-L_{\tau-})\ge 0$  by (\ref{eq5.prj}).
By predictability of $\tau$, we also have $E\mathbf{1}_{\{\tau<\infty\}}\Delta \Gamma^m_\tau=0$. Hence, by  (\ref{eq5.sap2}), $E\mathbf{1}_{\{\tau<\infty\}}\Delta \Gamma^{v,+}_\tau=0$. Therefore $P(\Pi(A_k))=0$ by (\ref{eq5.sap2}). In much the same way one can show that $P(\Pi(B_k))=0$. From this and Definition \ref{df2.dfsd}(c) we get (\ref{eq4.vth}).

Step 2. The general case.
Let   $(Y^{\overline\varepsilon},\Gamma^{\overline\varepsilon})$ be a unique  solution of the linear problem  RBSDE$^T(\xi,f(\cdot,Y),L,U+\varepsilon)$, and $(Y^{\underline\varepsilon},\Gamma^{\underline\varepsilon})$ be a unique solution of the linear problem
RBSDE$^T(\xi,f(\cdot,Y),L-\varepsilon,U)$ (in both cases $Y$ is frozen, where $Y$ is the solution to RBSDE$^T(\xi,f,L,U)$).
By Remark \ref{uw2.1} and Theorem \ref{tw4.ident}, $Y^{\underline\varepsilon}, Y^{\overline\varepsilon}$ are special semimartingales and $(Y^{\overline\varepsilon},\Gamma^{\overline\varepsilon,v},\Gamma^{\overline\varepsilon,m})$, $(Y^{\underline\varepsilon},\Gamma^{\underline\varepsilon,v},\Gamma^{\underline\varepsilon,m})$ are  semimartingale solutions. Moreover,
by Theorem \ref{tw.unq}, $Y^{\underline\varepsilon}\le Y\le Y^{\overline\varepsilon}$. Hence $\tau^{*,\underline\varepsilon}_\alpha\ge \tau^*_\alpha$ and $\sigma^{*,\overline\varepsilon}_\alpha\ge \sigma^*_\alpha$, where
\[
\tau^{*,\underline\varepsilon}_\alpha
=\inf\{t\ge\alpha:Y^{\underline\varepsilon}_t=U_t\}\wedge T,\quad
\sigma^{*,\overline\varepsilon}_\alpha
=\inf\{t\ge\alpha:Y^{\overline\varepsilon}_t=L_t\}\wedge T.
\]
By the first step (see (\ref{eq4.vth})),
\[
Y^{\overline\varepsilon}_t=Y^{\overline\varepsilon}_\alpha-\int_\alpha^t f(r,Y_r)\,dr+\int_\alpha^t\,d\Gamma^{\overline\varepsilon,v,-}_r
+\int_\alpha^t\,d\Gamma^{\overline\varepsilon,m}_r,\quad t\in [\alpha,\tau^*_\alpha\wedge\sigma^*_\alpha],
\]
and
\[
Y^{\underline\varepsilon}_t=Y^{\underline\varepsilon}_\alpha-\int_\alpha^t f(r,Y_r)\,dr-\int_\alpha^t\,d\Gamma^{\underline\varepsilon,v,+}_r
+\int_\alpha^t\,d\Gamma^{\underline\varepsilon,m}_r,\quad t\in [\alpha,\tau^*_\alpha\wedge\sigma^*_\alpha].
\]
Therefore $Y^{\overline\varepsilon}+\int_\alpha^\cdot f(r,Y_r)\,dr$ is a  submartingale of class (D) on $[\alpha,\tau^*_\alpha\wedge\sigma^*_\alpha]$ and $Y^{\underline\varepsilon}+\int_\alpha^\cdot f(r,Y_r)\,dr$ is a supermartingale of class (D) on $[\alpha,\tau^*_\alpha\wedge\sigma^*_\alpha]$. By Theorem \ref{tw.stb}, $Y^{\overline\varepsilon}+\int_\alpha^\cdot f(r,Y_r)\,dr\rightarrow Y+\int_\alpha^\cdot f(r,Y_r)\,dr$ and $Y^{\underline\varepsilon}+\int_\alpha^\cdot f(r,Y_r)\,dr\rightarrow Y+\int_\alpha^\cdot f(r,Y_r)\,dr$  in the norm $\|\cdot\|_{1;\alpha,\tau^*_\alpha\wedge\sigma^*_\alpha}$ as $\varepsilon\searrow 0$. It follows that  $Y+\int_\alpha^\cdot f(r,Y_r)\,dr$ is a uniformly integrable martingale on $[\alpha,\tau^*_\alpha\wedge\sigma^*_\alpha]$. From this one can deduce (\ref{eq5.abc}).
\end{proof}

\begin{uw}
\label{uw.locint}
Assume that (H1), (H2), (\ref{eq5.prj}) are satisfied.  Let $(Y,\Gamma)$ be a solution of {\rm RBSDE}$^T(\xi,f,L,U)$. Then for every $\alpha\in\mathcal T$,
\begin{equation}
\label{eq4.locint}
E\int_\alpha^{\sigma^*_\alpha\wedge\tau^*_\alpha}|f(r,Y_r)|\,dr<\infty,
\end{equation}
where $\sigma^*_\alpha, \tau^*_\alpha$ are defined  by (\ref{eq5.sp}).
Moreover, the process $Y+\int_\alpha^\cdot f(r,Y_r)\,dr$ is a uniformly integrable martingale on the closed interval $[\alpha,\sigma^*_\alpha\wedge\tau^*_\alpha]$. To see this, we set
\[
\tau_k=\inf\{t\ge \alpha: \int_\alpha^t|f(r,Y_r)|\,dr\ge k\}\wedge T,\qquad \theta_k=\sigma^*_\alpha\wedge\tau^*_\alpha\wedge\tau_k.
\]
By Theorem \ref{tw.sp}, the process $Y+\int_\alpha^\cdot f(r,Y_r)\,dr$ is a uniformly integrable martingale on $[\alpha,\theta_k]$. Therefore $Y$ is the first component of the solution of BSDE$^{\alpha,\theta_k}(Y_{\theta_k},f)$. By (H1), (H2) and \cite[Theorem 2.9]{K:arx},
\[
E\int_\alpha^{\theta_k}|f(r,Y_r)|\,dr\le E|Y_{\theta_k}|+E|S_{\theta_k}|+E\int_\alpha^{\theta_k}|f(r,S_r)|\,dr+E\int_\alpha^{\theta_k}\,d|S^v|_r,
\]
where $S^v$ is the predictable finite variation part of the Doob-Meyer decomposition of $S$. Letting $k\rightarrow\infty$  and using  (H1), (H2) and the fact that $Y$ is of class (D) yields (\ref{eq4.locint}). From this we  easily conclude that the process $Y+\int_\alpha^\cdot f(r,Y_r)\,dr$ is a uniformly integrable martingale on $[\alpha,\sigma^*_\alpha\wedge\tau^*_\alpha]$.
\end{uw}

\begin{uw}
\label{uw.stt}
By Theorem \ref{tw.obs} and Remark \ref{uw.dps}, the value process $Y$ in the Dynkin game (\ref{eqi.1}) can be approximated by solutions $Y^n$ of the  penalized equation (\ref{eqi.5.5}). This kind of results had appeared in the literature much before the notion of reflected BSDEs was introduced. In \cite{Stettner3} (see also \cite{Stettner,Stettner1} for Markovian case) Stettner
proved that $Y$ given by (\ref{eqi.1}), but with $f\equiv 0$, $T=\infty$ and barriers of the following special form
\[
L_t=e^{- a t} \hat L_t,\quad U_t=e^{-at}\hat U_t,\quad t\ge 0,
\]
where $a>0$ and $\hat L,\hat U$ are bounded right-continuous adapted processes, can by approximated by
solutions of the following equation
\begin{align}
\label{eqi.7}
Y^n_t=nE\Big( \int_t^\infty (Y^n_r-L_r)^-\,dr\nonumber-\int_t^\infty (Y^n_r-U_r)^+\,dr|\FF_t\Big).
\end{align}
Observe that if we define $M^n$ as
\begin{align*}
M^n_t=nE\Big( \int_0^\infty (Y^n_r-L_r)^-\,dr\nonumber-\int_0^\infty (Y^n_r-U_r)^+\,dr|\FF_t\Big)-Y^n_0,
\end{align*}
then the pair $(Y^n,M^n)$ is a solution of the penalized BSDE (\ref{eqi.5.5}) with $f(r,y)\equiv 0$,
$\xi=0$ and $T=\infty$.
\end{uw}

\subsection{Nonlinear $f$-expectation and generalized Dynkin games}

We now  introduce the notion of the nonlinear expectation
\[
\EE^f_{\alpha,\beta}: L^1(\Omega,\FF_\beta;P)\rightarrow L^1(\Omega,\FF_\alpha;P)
\]
for $\alpha,\beta\in\TT$ such that $\alpha\le\beta$ and for $f$ satisfying (H1)--(H4) with $\mu\le 0$. For  $\xi\in L^1(\Omega,\FF_\beta;P)$, we put
\[
\EE^f_{\alpha,\beta}(\xi)= Y_\alpha,
\]
where $(Y,M)$ is the unique solution of BSDE$^\beta(\xi,f)$.

We say that a c\`adl\`ag process $X$ of class (D) is an $\EE^f$-supermartingale
(resp. $\EE^f$-submartingale) on $[\alpha,\beta]$
if $\EE^f_{\sigma,\tau}(X_\tau)\le X_\sigma$ (resp. $\EE^f_{\sigma,\tau}(X_\tau)\ge X_\sigma$) for all
$\tau,\sigma\in \TT$ such that $\alpha\le\sigma\le\tau\le\beta$.
Of course, $X$ is called an $\EE^f$-martingale on $[\alpha,\beta]$ if it is both $\EE^f$-supermartingale and $\EE^f$-submartingale on $[\alpha,\beta]$. For a given c\`adl\`ag process $V$ and stopping times $\alpha,\beta$ ($\alpha\le \beta$) we denote  by $|V|_{\alpha,\beta}$ the total variation of the process $V$ on $[\alpha,\beta]$.

\begin{stw}
\label{stw.fexp}
Assume that  $f$ satisfies \mbox{\rm(H1)--(H4)} with $\mu\le 0$ and $\alpha,\beta\in\TT$, $ \alpha\le\beta$.
\begin{enumerate}
\item[\rm(i)] Let $\xi\in L^1(\Omega,\FF_\beta;P)$ and $V$ be a c\`adl\`ag $\BF$-adapted finite
variation process such that  $V_\alpha=0$ and $E|V|_{\alpha,\beta}<\infty$. Then there exists a unique solution $(X,N)$
of  BSDE$^{\alpha,\beta}(\xi,f+dV)$.
Moreover,  if $V$ (resp. $-V$) is an increasing process, then $X$ is an $\EE^f$-supermartingale (resp. $\EE^f$-submartingale) on $[\alpha,\beta]$.
\item[\rm(ii)] If  $\xi_1,\xi_2\in L^1(\Omega,\FF_\beta;P)$ and $\xi_1\le\xi_2$, then $\EE^f_{\alpha,\beta}(\xi_1)\le\EE^f_{\alpha,\beta}(\xi_2)$.
\item[\rm(iii)] If $f_1,f_2$ satisfy \mbox{\rm(H1)--(H4)} with $\mu\le 0$, $\alpha,\beta_1,\beta_2\in\TT$, $\alpha\le\beta_1\le\beta_2$, $\xi_1\in L^1(\Omega,\FF_{\beta_1};P)$, $\xi_2\in L^1(\Omega,\FF_{\beta_2};P)$ then
\begin{align*}
|\EE^{f_1}_{\alpha,\beta_1}(\xi_1)-\EE^{f_2}_{\alpha,\beta_2}(\xi_2)|&\le E\Big(|\xi_1-\xi_2|+\int_\alpha^{\beta_1}|f^1(r,Y^1_r)-f^2(r,Y^1_r)|\,dr \\ &\quad+\int_{\beta_1}^{\beta_2}|f^2(r,Y^2_r)|\,dr|\FF_\alpha\Big),
\end{align*}
where $Y^1_t=\EE^{f^1}_{t\wedge\beta_1,\beta_1}(\xi_1)$, $Y^2_t=\EE^{f^2}_{t\wedge\beta_2,\beta_2}(\xi_2)$.
\end{enumerate}
\end{stw}
\begin{proof}
Assertion (iii) follows from Theorem \ref{tw.stb} and (ii) follows from Theorem \ref{tw.unq}. The existence part in (i)
follows from \cite[Theorem 2.9]{K:arx}. Now assume that $X$ is as in (i) and $V$ is an increasing process. Let $\sigma,\tau\in\TT$ be such that  $\alpha\le\sigma\le\tau\le\beta$, and let $(X^\tau,N^\tau)$ be a solution of BSDE$^{\alpha,\tau}(X_\tau,f)$. It is clear that $(X,N)$ is a solution of
BSDE$^{\alpha,\tau}(X_\tau,f+dV)$. Therefore, by Theorem \ref{tw.unq}, $X\ge X^\tau$ on $[\alpha,\tau]$. In particular,
$X_\sigma\ge X^\tau_\sigma$.  By the definition of the nonlinear expectation, $\EE^f_{\sigma,\tau}(X_\tau)=X^\tau_\sigma$, so $\EE^f_{\sigma,\tau}(X_\tau)\le X_\sigma$. A similar  reasoning  in the case where $-V$ is increasing gives the result.
\end{proof}

\begin{tw}
Assume that \mbox{\rm(H1)--(H4)} are satisfied with $\mu\le 0$.
\begin{enumerate}
\item[\rm{(i)}]
$(Y,\Gamma)$ is a solution of {\rm RBSDE}$^T(\xi,f,L,U)$ if and only if  for every $\alpha\in\TT$,
\begin{align}
\label{eq4.5f}
Y_\alpha= \esssup_{\sigma\ge\alpha}\essinf_{\tau\ge\alpha}J^f_\alpha(\tau,\sigma)
=\essinf_{\tau\ge\alpha}\esssup_{\sigma\ge\alpha}J^f_\alpha(\tau,\sigma),
\end{align}
where
\begin{align}
\label{eq.vff}
J^f_\alpha(\tau,\sigma)=\EE^f_{\alpha,\tau\wedge\sigma}
(L_\sigma\mathbf{1}_{\sigma<\tau}+U_\tau\mathbf{1}_{\tau\le\sigma<T}+\xi\mathbf{1}_{\sigma=\tau=T}).
\end{align}
\item[\rm{(ii)}]
Let $(Y,\Gamma)$ be a solution of {\rm RBSDE}$^T(\xi,f,L,U)$. Then for all $\sigma,\tau\in \TT_\alpha$ we have
\begin{equation}
\label{eq5.aaggf}
J^f_\alpha(\tau_\varepsilon,\sigma)-\varepsilon\le Y_\alpha\le J^f_\alpha(\tau,\sigma_\varepsilon)+\varepsilon,
\end{equation}
where $\tau_\varepsilon, \sigma_\varepsilon$ are defined by \mbox{\rm(\ref{eq5.4})}.
\end{enumerate}
\end{tw}
\begin{proof}
The proof of (ii) and the necessity part of  (i)  is similar to the proof of Theorem \ref{tw.obs}. The only difference is that the sequence $\{\delta_n\}$
defined in that proof should now  satisfy the  additional condition $E\int_\alpha^{\delta_n}|f(r,Y_r)|\,dr+E\int_\alpha^{\delta_n}\,d\Gamma^{v,+}_r(\alpha)<\infty$. Indeed, by (\ref{eq5.ccaa}) and Proposition \ref{stw.fexp}(i), we get
\[
\EE^f_{\alpha,\theta_k}(Y_{\theta_k})\le Y_\alpha.
\]
By the reasoning following (\ref{eq5.ccaa}), we know that $Y_{\theta_k}\rightarrow Y_{\tau_\varepsilon\wedge \sigma}$ as  $k\rightarrow \infty$. So, by Proposition \ref{stw.fexp}(iii) and the above inequality
\[
\EE^f_{\alpha,\tau_\varepsilon\wedge \sigma}(Y_{\tau_\varepsilon\wedge \sigma})\le Y_\alpha.
\]
By the definition of $\tau_\varepsilon$ and Proposition \ref{stw.fexp}(ii), we conclude from the above inequality
\[
\EE^f_{\alpha,\tau_\varepsilon\wedge\sigma}
(L_\sigma\mathbf{1}_{\sigma<\tau_\varepsilon}+U_{\tau_\varepsilon}\mathbf{1}_{\tau_\varepsilon\le\sigma<T}-\varepsilon+\xi\mathbf{1}_{\sigma=\tau_\varepsilon=T})\le Y_\alpha.
\]
From this and Proposition \ref{stw.fexp}(iii), we get the left-hand side inequality in (\ref{eq5.aaggf}).
An analogous reasoning applied to the pair $(\sigma_\varepsilon,\tau)$ gives the right-hand side inequality in (\ref{eq5.aaggf}). From (\ref{eq5.aaggf}), we easily get  (\ref{eq4.5f}). For the  sufficiency in (i) let us denote
by $G_\alpha$ the right-hand side of (\ref{eq4.5f}). By Theorem \ref{tw5.2} there exists a unique solution $(Y,\Gamma)$
to RBSDE$^T(\xi,f,L,U)$. By the necessity part in (i), $G_\alpha=Y_\alpha,\, \alpha\in \TT$, so $(G,\Gamma)$
is a solution to RBSDE$^T(\xi,f,L,U)$.
\end{proof}

\begin{stw}
\label{tw.spf}
Assume that \mbox{\rm (H1)--(H4)} with $\mu\le0$ hold true  and \mbox{\rm(\ref{eq5.prj})} is satisfied.
Let  $(Y,\Gamma)$ be a solution of {\rm RBSDE}$^T(\xi,f,L,U)$.  Then
for every $\alpha\in\TT$,
\begin{equation}
\label{eq5.abcf}
Y_\alpha=J^f_\alpha(\sigma^*_\alpha,\tau^*_\alpha),
\end{equation}
where $J^f_\alpha$ is defined by \mbox{\rm(\ref{eq.vff})} and $\sigma^*_\alpha, \tau^*_\alpha$ are defined by \mbox{\rm(\ref{eq5.sp})}.
\end{stw}
\begin{proof}
By Remark \ref{uw.locint} and  Proposition \ref{stw.fexp}(i), $\EE^f_{\alpha,\sigma^*_\alpha\wedge\tau^*_\alpha}(Y_{\sigma^*_\alpha\wedge\tau^*_\alpha})=Y_\alpha$.
From this  we get (\ref{eq5.abcf}).
\end{proof}

\section{Existence result for RBSDEs: the general result}
\label{sec7}

Our main existence theorem for RBSDEs (Theorem \ref{tw5.2}) says that under (H1)--(H4) there exists a unique solution
to RBSDE$^T(\xi,f,L,U)$. In this section, we will show that under the assumption that $L,U$ are of class (D) (not merely $L^+, U^-$)
one can dispense with condition (\ref{eq1.11}), which is part of (H1), and still  get an existence result for RBSDE$^T(\xi,f,L,U)$.
This result is interesting since it
implies
that there exist data $(\xi,f,T)$ such that there is no solution of BSDE$^T(\xi,f)$ and, at the same time, for every $\mathbb F$-adapted c\`adl\`ag barriers $L,U$ of class (D)  ($L\le U$) there exists a solution to RBSDE$^T(\xi,f,L,U)$ (see Example \ref{prz4.1.09}). As in the previous section $T$ is an arbitrary stopping time.

\begin{stw}
\label{ex.gen}
Assume that $E|\xi|<\infty$,  \mbox{\rm(H2)--(H4)} are satisfied with $\mu\le0$. Then there exists a  unique solution of  {\rm RBSDE}$^T(\xi,f,L,U)$.
\end{stw}
\begin{proof}
We only need to prove the existence of a solution. To this end, we  write
\[
f_{n,m}(t,y)=(f(t,y)\wedge n\eta_t)\vee (-m\eta_t).
\]
By Theorem \ref{tw5.2} there exists a unique solution of RBSDE$^T(\xi,f_{n,m},L,U)$.
By Theorem \ref{tw.unq}, $Y^{n,m}\le Y^{n+1,m}$ and $Y^{n,m}\ge Y^{n,m+1}$. We put $Y^m=\lim_{n\rightarrow \infty} Y^{n,m}$.
Of course, $L\le Y^m\le U$, so $Y^m$ is of class (D).  Next, we observe that $Y^m\ge Y^{m+1}$ and we  put $Y=\lim_{m\rightarrow \infty} Y^m$. Of course, $L\le Y\le U$, so $Y$ is of class (D). By the definition,
\[
Y^{n,m}_t=Y^{n,m}_{T\wedge a}+\int_t^{T\wedge a} f_{n,m}(r,Y^{n,m}_r)\,dr+\int_t^{T\wedge a}\,d\Gamma^{n,m}_r,\quad t\in [0,T\wedge a].
\]
Since $L\le Y^{n,m}\le U$, letting $n\rightarrow \infty$ and then  $m\rightarrow \infty$
in the above equation and using  (H2)--(H4) we obtain
\[
Y_t=Y_{T\wedge a}+\int_t^{T\wedge a} f(r,Y_r)\,dr+\int_t^{T\wedge a}\,d\Gamma_r,\quad t\in [0,T\wedge a].
\]
Since $L,U$ are of class (D), by (H4) there exists a chain $\{\tau_k\}$ on $[0,T]$ such that
\[
E\int_0^{\tau_k}|f(r,L_r)|\,dr+E\int_0^{\tau_k}|f(r,U_r)|\,dr<\infty.
\]
From what has already been proved and (H2)--(H4) it follows that
\begin{equation}
\label{eq5.2}
\lim_{m\rightarrow \infty}\lim_{n\rightarrow \infty}E\int_0^{\tau_k}|f_{n,m}(r,Y^{n,m}_r)-f(r,Y_r)|\,dr=0.
\end{equation}
Hence, by Theorem \ref{tw.stb}, $\lim_{m\rightarrow \infty}\lim_{n\rightarrow \infty} \|Y^{n,m}-Y\|_{1;\tau_k}=0$. Therefore  $Y$ is a c\`adl\`ag process and $Y_{\tau_k\wedge a}\rightarrow Y_{\tau_k}$ as $a\rightarrow \infty$. Since $\{\tau_k\}$ is a chain on $[0,T]$, we get (\ref{eq5.1tutu}).
By Theorem \ref{tw.obs},
\begin{align*}
Y^{n,m}_\alpha=\esssup_{\tau_k\wedge a\ge \sigma\ge\alpha}\essinf_{\tau_k\wedge a\ge \tau\ge\alpha} & E\Big(\int_\alpha^{\tau\wedge\sigma}f_{n,m}(r,Y^{n,m}_r)\,dr \nonumber\\
&+L_\sigma\mathbf{1}_{\sigma<\tau}+U_\tau\mathbf{1}_{\tau\le\sigma<\tau_k\wedge a}+Y^{n,m}_{\tau_k\wedge a}\mathbf{1}_{\sigma=\tau=\tau_k\wedge a}|\FF_\alpha\Big).
\end{align*}
Letting  $n\rightarrow \infty$ and then $m\rightarrow \infty$ and using  (\ref{eq5.2}) we obtain
\begin{align*}
Y_\alpha=\esssup_{\tau_k\wedge a\ge \sigma\ge\alpha}\essinf_{\tau_k\wedge a\ge \tau\ge\alpha} & E\Big(\int_\alpha^{\tau\wedge\sigma}f(r,Y_r)\,dr \nonumber\\
&+L_\sigma\mathbf{1}_{\sigma<\tau}+U_\tau\mathbf{1}_{\tau\le\sigma<\tau_k\wedge a}
+Y_{\tau_k\wedge a}\mathbf{1}_{\sigma=\tau=\tau_k\wedge a}|\FF_\alpha\Big).
\end{align*}
By Corollary \ref{wn.viwn},   $(Y,\Gamma)$ is a  solution of RBSDE$^{\tau_k\wedge a }(Y_{\tau_k\wedge a},f,L,U)$. Since $\{\tau_k\}$ is a chain, we conclude that the pair $(Y,\Gamma)$ is a solution of the equation RBSDE$^{T}(\xi,f,L,U)$.

\end{proof}

\begin{prz}
\label{prz4.1.09}
Let $c$ be a positive $\mathbb F$-adapted  c\`adl\`ag   process such that $E\int_0^Tc_t\,dt=\infty$. Assume that $T$ is bounded.
We put $f(t,y)=-c_t(1+y^+)$ and $\xi\equiv 1$. Then, by Proposition \ref{ex.gen},
for all c\`adl\`ag $L,U$ of class (D) ($L\le U$) there exists a unique solution to RBSDE$^T(\xi,f,L,U)$. However, there is no solution to BSDE$^T(\xi,f)$. Indeed, assume that $(Y,M)$ is a solution to BSDE$^T(\xi,f)$. Then, in particular, $Y$ is of class (D). Let $\{\tau_k\}$ be a fundamental sequence for the local martingale $M$. By the definition of a solution to BSDE$^T(\xi,f)$,
\[
Y_0=Y_{\tau_k}+\int_0^{\tau_k} f(r,Y_r)\,dr-\int_0^{\tau_k}\,dM_r.
\]
Since $-f$ is positive, we have
\[
\int_0^{\tau_k} |f(r,Y_r)|\,dr=-Y_0+Y_{\tau_k}-\int_0^{\tau_k}\,dM_r.
\]
Taking the expectation of  both sides of the equation and using the fact that $\{\tau_k\}$ is a fundamental sequence for $M$ yields
$
E\int_0^{\tau_k} |f(r,Y_r)|\,dr=-EY_0+EY_{\tau_k}.
$
Since $Y$ is of class (D), we get
\[
E\int_0^T |f(r,Y_r)|\,dr\le 2\|Y\|_{1;T}<\infty.
\]
On the other hand,
\[
E\int_0^T |f(r,Y_r)|\,dr\ge E\int_0^T c_r\,dr=\infty.
\]
\end{prz}

\appendix

\section{Semimartingale solutions to RBSDEs}\label{app}
\label{sec2}
In this section, we collect and extend the results on semimartingale solutions to RBSDEs on general filtered spaces.
As in the whole paper,  we  consider  only reflected BSDEs  with generator $f$ independent of the martingale part of solution.
However, in the case of $L^2$ data, based on the results of this section, one can easily get the existence result in general case by the fixed point theorem (see \cite[Section 6]{K:SPA2}).

In the whole section, we assume that $L,U$ satisfy Mokobodzki's condition.
Therefore, according to the terminology of the paper, we are dealing with semimartingale solutions of RBSDEs.
As it is not standard terminology used in the literature (we introduced it here to distinguish our new kind of solutions from that existing in the literature), we keep it only in the definitions, and in all the results we drop the word ``semimartingale".

Let
$\alpha,\beta\in \mathcal T$ be  such that $\alpha\le\beta$. In what follows  $\hat\xi$ is an $\FF_\beta$-measurable random variable and
$V$ is an $\BF$-adapted
c\`adl\`ag process of finite variation on $[\alpha,\beta]$ with $V_\alpha=0$.

\subsection{Semimartingale solutions with bounded terminal time}
\label{sec2.1}

In this section, we assume that $T$ is a bounded stopping time. 

\begin{df}
\label{def.A1}
We say that a pair $(Y,M)$ of $\BF$-adapted c\`adl\`ag processes is a  solution of the  backward stochastic differential equation on $[\alpha,\beta]$ with terminal condition $\hat \xi$, generator $f+dV$ ({\rm BSDE}$^{\alpha,\beta}(\hat\xi,f+dV)$ for short) if
\begin{enumerate}
\item[(i)] $Y$ is of class (D) on $[\alpha,\beta]$, $M$ is a local martingale on $[\alpha,\beta]$
with $M_\alpha=0$,
\item[(ii)] $\int_\alpha^\beta|f(r,Y_r)|\,dr<\infty$ and
\[
Y_t=\hat\xi+\int_t^\beta f(r,Y_r)\,dr+\int_t^\beta \,dV_r-\int_t^\beta\,dM_r,\quad t\in [\alpha,\beta].
\]
\end{enumerate}
\end{df}

\begin{df}
\label{df2.2}
We say that a triple $(Y,M,K)$ of $\BF$-adapted c\`adl\`ag processes is a (semimartingale) solution of  reflected backward stochastic differential equation on $[\alpha,\beta]$ with terminal condition
$\hat \xi$, generator $f$ and lower barrier $L$ ($\underline{\rm{R}}${\rm BSDE}$^{\alpha,\beta}(\hat\xi,f,L)$ for short) if
\begin{enumerate}
\item[(i)] $Y$ is of class (D) on $[\alpha,\beta]$, $K$ is an increasing predictable process on $[\alpha,\beta]$ with $K_\alpha=0$, $M$ is a local martingale on $[\alpha,\beta]$
with $M_\alpha=0$,
\item[(ii)] $\int_\alpha^\beta|f(r,Y_r)|\,dr<\infty$ and
\[
Y_t=\hat\xi+\int_t^\beta f(r,Y_r)\,dr+\int_t^\beta\,dK_r-\int_t^\beta\,dM_r,\quad t\in [\alpha,\beta],
\]
\item[(iii)] $L_t\le Y_t,\, t\in [\alpha,\beta]$, and
\[
\int_\alpha^{\beta}(Y_{r-}-L_{r-})\,dK_r=0.
\]
\end{enumerate}
\end{df}

\begin{df}
We say that a triple $(Y,M,A)$ of $\BF$-adapted c\`adl\`ag processes is a (semimartingale) solution of  reflected backward stochastic differential equation on $[\alpha,\beta]$ with terminal condition
$\hat \xi$, generator $f$ and upper barrier $U$ ($\overline{\rm{R}}${\rm BSDE}$^{\alpha,\beta}(\hat\xi,f,U)$ for short) if $(-Y,A,-M)$ is a (semimartingale) solution to $\underline{\rm{R}}${\rm BSDE}$^{\alpha,\beta}(-\hat\xi,\tilde f,-U)$, where
$\tilde f(t,y)=-f(t,-y)$.
\end{df}

Let us consider the following hypotheses:
\begin{enumerate}
\item[(A1)] $\hat\xi$ is $\FF_\beta$-measurable, $E|\hat \xi|<\infty$  and there exists a c\`adl\`ag process $S$, which is a difference of supermartingales of class (D), such that
\[
E\int_\alpha^\beta|f(r,S_r)|\,dr<\infty.
\]
\item[(A2)]  There exists $\mu\in\BR$ such that for a.e. $t\in [\alpha,\beta]$.
\[
(y-y')(f(t,y)-f(t,y'))\le\mu |y-y'|^2,\quad y,y'\in\BR,\quad \mbox{a.s.}
\]

\item[(A3)]For a.e. $t\in [\alpha,\beta]$ the function $y\mapsto f(t,y)$ is continuous a.s.

\item[(A4)] For every $y\in\BR$, $\int_\alpha^\beta|f(r,y)|\,dr<\infty$ a.s.
\end{enumerate}

\begin{stw}
\label{tw2.0}
Assume that $\hat\xi,f$ satisfy \mbox{\rm(A1)--(A4)}, $L_\beta\le\hat\xi$ and $L^+$ is of class \mbox{\rm(D)} on $[\alpha,\beta]$.
\begin{enumerate}
\item[\rm(i)]There exists a unique  solution $(Y,M,K)$ of
${\rm{\underline{R}BSDE}}^{\alpha,\beta}(\hat\xi,f,L)$.

\item[\rm(ii)] Let $\{\hat\xi_n\}$ be a sequence of  integrable $\FF_\beta$-measurable random variables
such that $\hat\xi_n\nearrow\hat\xi$, and let $(Y^n,M^n)$ be a  solution of {\rm BSDE}$^{\alpha,\beta}(\hat\xi_n,f_n)$ with $f_n(t,y)=f(t,y)+n(y-L_t)^-$. Then
$Y^n_t\nearrow Y_t$, $t\in [\alpha,\beta]$.
\end{enumerate}

\end{stw}
\begin{proof}
By \cite[Theorem 2.7]{KR:JFA}, there exists a unique solution $(\hat Y,\hat M)$ of BSDE$^{\alpha,\beta}(\hat\xi,f)$. Since $\hat Y\vee L$
is of class (D) on $[\alpha,\beta]$, by \cite[Theorem 4.1]{K:SPA2} there exists a unique solution $(Y,M,K)$ of
${\rm{\underline{R}BSDE}}^{\alpha,\beta}(\hat\xi,f,L\vee \hat Y)$.
Let $(\hat Y^n,\hat M^n)$ be a solution of BSDE$^{\alpha,\beta}(\hat\xi^n,f^n)$ with $\hat\xi^n=\hat\xi\wedge (-n)$ and $f^n=f\wedge (-n)$. By \cite[Theorem 4.1]{K:SPA2}, there exists a unique solution $(\bar Y^n,\bar M^n,\bar K^n)$ of
${\rm{\underline{R}BSDE}}^{\alpha,\beta}(\hat\xi,f,L\vee \hat Y^n)$. Furthermore, by \cite[Proposition 2.1]{K:SPA2}, $\hat Y^n\le\hat Y\le \bar Y^n$,
which implies that $\hat Y^n\vee L\le \hat Y\vee L\le \bar Y^n$. Thus $(\bar Y^n,\bar M^n,\bar K^n)$ is a solution of
${\rm{\underline{R}BSDE}}^{\alpha,\beta}(\hat\xi,f,L\vee \hat Y)$. Consequently, by uniqueness (see \cite[Corollary 2.2]{K:SPA2}),
$(\bar Y^n,\bar M^n,\bar K^n)=(Y,M,K)$, $n\ge 1$. In particular,  for any $n\ge1$,
\[
\int_\alpha^\beta (Y_{r-}-L_{r-}\vee \hat Y^n_{r-})\,dK_r=0.
\]
Letting $n\rightarrow \infty$ we get $\int_\alpha^\beta (Y_{r-}-L_{r-})\,dK_r=0$. Since $Y\ge L$, we see that in fact  $(Y,M,K)$ is  a solution to ${\rm{\underline{R}BSDE}}^{\alpha,\beta}(\hat\xi,f,L)$. We may now repeat step by step the proof of \cite[Theorem 4.1]{K:SPA2}, with obvious changes, to show the convergence of $\{Y^n\}$.
\end{proof}

\begin{uw}
\label{uw.kd}
In  the proof of Proposition \ref{tw2.0} we have showed that under (A1)--(A4) a triple $(Y,M,K)$ is a solution of
${\rm{\underline{R}BSDE}}^{\alpha,\beta}(\hat\xi,f,L)$ if and only if  it is a solution of ${\rm{\underline{R}BSDE}}^{\alpha,\beta}(\hat\xi,f,L\vee \hat Y)$, where $(\hat Y,\hat M)$  is the solution of {\rm BSDE}$^{\alpha,\beta}(\hat\xi,f)$. Therefore, without loss of generality, one can assume that $L$ is of class (D) (and not merely that $L^+$ is of class (D)).
\end{uw}

\begin{df}
\label{df2.dfsd}
We say that a triple $(Y,M,R)$ of $\BF$-adapted c\`adl\`ag processes is a (semimartingale) solution to reflected {\rm BSDE} on $[\alpha,\beta]$ with terminal condition
$\hat \xi$, generator $f$, lower barrier $L$ and upper barrier $U$ ({\rm RBSDE}$^{\alpha,\beta}(\hat\xi,f,L,U)$ for short) if
\begin{enumerate}
\item[(a)] $Y$ is of class (D) on $[\alpha,\beta]$, $R$ is a finite variation predictable process on $[\alpha,\beta]$ with $R_\alpha=0$, $M$ is a local martingale on $[\alpha,\beta]$ with $M_\alpha=0$,
\item[(b)] $\int_\alpha^\beta|f(r,Y_r)|\,dr<\infty$ and
\[
Y_t=\hat\xi+\int_t^\beta f(r,Y_r)\,dr+\int_t^\beta\,dR_r-\int_t^\beta\,dM_r,\quad t\in [\alpha,\beta],
\]
\item[(c)] $L_t\le Y_t\le U_t,\, t\in [\alpha,\beta]$, and
\[
\int_\alpha^{\beta}(Y_{r-}-L_{r-})\,dR^+_r=\int_\alpha^{\beta}(U_{r-}-Y_{r-})\,dR^-_r=0.
\]
\end{enumerate}
\end{df}
If $\alpha=0$, we write RBSDE$^\beta$ instead of RBSDE$^{0,\beta}$.

\begin{stw}
\label{tw2.1}
Assume that $\hat\xi,f$ satisfy \mbox{\rm(A1)--(A4)}, $L_\beta\le\hat\xi\le U_\beta$ and $L^+, U^-$ are of class \mbox{\rm(D)} on $[\alpha,\beta]$.
\begin{enumerate}
\item[\rm(i)] There exists a solution $(Y,M,R)$ of {\rm{RBSDE}}$^{\alpha,\beta}(\hat\xi,f,L,U)$ if and only if
there exists a special semimartingale $X$ such that $L_t\le X_t\le U_t,\, t\in [\alpha,\beta]$.

\item[\rm(ii)] Let $\{\hat\xi_n\}$ be a sequence of $\FF_\beta$-measurable
integrable random variables such that $\hat\xi_n\nearrow\hat\xi$, and let
$(\bar Y^n,\bar A^n,\bar M^n),\, n\ge 1,$ be a solution of $\overline{{\rm{R}}}${\rm{BSDE}}$^{\alpha,\beta}(\hat\xi_n,f_n,U)$
with
\[
f_n(t,y)=f(t,y)+n(y-L_t)^-.
\]
Then $\bar Y^n_t\nearrow Y_t,\, t\in [\alpha,\beta]$.
\end{enumerate}
\end{stw}
\begin{proof}
Of course, if there exists a solution $(Y,M,R)$ of the problem RBSDE$^{\alpha,\beta}(\hat\xi,f,L,U)$, then $Y$
is a special semimartingale which lies between the barriers. Suppose  now that there exists a special semimartingale $X$ such that $L_t\le X_t\le U_t$, $t\in [\alpha,\beta]$. To show the  existence of a solution it suffices to modify slightly the proof of \cite[Theorem 4.2]{K:SPA2}. Indeed, in \cite{K:SPA2} the existence of a solution of  RBSDE$^{\alpha,\beta}(\hat\xi,f,L,U)$ is proved
under the additional assumption that $E\int_\alpha^\beta\,d |V|_r<\infty$, where $V$
is the finite variation part from the Doob-Meyer decomposition of $X$, and $L,U$ are of class (D). However, the proof of \cite[Theorem 4.2]{K:SPA2} applies also to our case. The only difference is that in the present situation the sequence $\{\delta_k\}$ appearing in the proof of \cite[Theorem 4.2]{K:SPA2}
should be defined as follows:
\[
\delta_k=\inf\{t\ge\beta:\int_\alpha^t|f(r,X_r)|\,dr\ge k\}\wedge\sigma_k,
\]
where $\{\sigma_k\}$  is a chain on $[\alpha,\beta]$ such that $E\int_\alpha^{\sigma_k}\,d |V|_r<\infty$. Such a chain exists since $V$ is predictable (and $X_\alpha$ is integrable). The fact that $L,U$ are of class (D) was used in the proof of \cite[Theorem 4.2]{K:SPA2}  only
to apply  \cite[Theorem 2.13]{K:SPA2} to some reflected BSDE with upper barrier $U$. However, we have shown in Proposition \ref{tw2.0} that \cite[Theorem 2.13]{K:SPA2} is still true when we only assume that $U^-$ is of class (D). The proof of part (ii) runs as the proof of  \cite[Theorem 4.2]{K:SPA2} with obvious changes (in \cite[Theorem 4.2]{K:SPA2} the case of  terminal conditions not depending on $n$ is considered).
\end{proof}

\begin{uw}
\label{uw2.1}
If $L_t<U_t,\, t\in [\alpha,\beta]$, and $L_{t-}< U_{t-}$, $t\in (\alpha,\beta]$, then one can easily show
that  there exists a special semimartingale $X$ such that $L_t\le X_t\le U_t,\, t\in [\alpha,\beta]$ (see \cite{T}).
\end{uw}

\subsection{Semimartingale solutions with arbitrary terminal time}

In this section, we extend some results of the previous section by dropping the assumption
that $T$ is bounded.

\begin{df}
We say that a pair $(Y,M)$ of $\BF$-adapted c\`adl\`ag processes is a solution of the backward stochastic
differential equation on $[\alpha,\beta]$ with terminal condition
$\hat \xi$, generator $f$ ({\rm BSDE}$^{\alpha,\beta}(\hat\xi,f)$ for short) if for every $a\ge0$ it is a solution of {\rm BSDE}$^{\alpha,(\beta\wedge a)\vee\alpha}(Y_{(\beta\wedge a)\vee\alpha},f)$ and
\begin{equation}
\label{eq5.1}
Y_{(\beta\wedge a)\vee\alpha}\rightarrow \hat \xi\quad\mbox{a.s. as }a\rightarrow \infty.
\end{equation}
\end{df}

\begin{df}
We say that a triple $(Y,M,K)$ of $\BF$-adapted c\`adl\`ag processes is a (semimartingale) solution of the reflected {\rm BSDE} on $[\alpha,\beta]$ with an $\FF_\beta$-measurable terminal condition
$\hat \xi$, generator $f$ and lower barrier $L$ ($\underline{\rm{R}}${\rm BSDE}$^{\alpha,\beta}(\hat\xi,f,L)$ for short) if for every $a\ge0$ it is a (semimartingale) solution of  $\underline{\rm{R}}${\rm BSDE}$^{\alpha,(\beta\wedge a)\vee\alpha}(Y_{(\beta\wedge a)\vee\alpha},f,L)$ and (\ref{eq5.1}) is satisfied.
\end{df}

\begin{df}
We say that a triple $(Y,M,A)$ of $\BF$-adapted c\`adl\`ag processes is a (semimartingale) solution of the reflected backward stochastic differential equation on $[\alpha,\beta]$ with  terminal condition
$\hat \xi$, generator $f$ and upper barrier $U$ ($\overline{\rm{R}}${\rm BSDE}$^{\alpha,\beta}(\hat\xi,f,U)$ for short) if $(-Y,A,-M)$ is a (semimartingale) solution of $\underline{\rm{R}}${\rm BSDE}$^{\alpha,\beta}(-\hat\xi,\tilde f,-U)$ with
\[
\tilde f(t,y)=-f(t,-y).
\]
\end{df}

\begin{df}
We say that a triple $(Y,M,R)$ of $\BF$-adapted c\`adl\`ag processes is a (semimartingale) solution of reflected backward stochastic differential equation on $[\alpha,\beta]$ with terminal condition
$\hat \xi$, generator $f$ and barriers $L$ and $U$ ({\rm RBSDE}$^{\alpha,\beta}(\hat\xi,f,L,U)$ for short) if (\ref{eq5.1}) is satisfied, and for any $a\ge0$,  $(Y,M,R)$ is a (semimartingale) solution to {\rm RBSDE}$^{\alpha,(\beta\wedge a)\vee\alpha}(Y_{(\beta\wedge a)\vee\alpha},f,L,U)$.
\end{df}

\begin{uw}
If $\beta<\infty$, then the above definitions are equivalent to the corresponding definitions of Section \ref{sec2.1}.
\end{uw}

\begin{uw}
\label{uw.red}
Let $S,T$ be  stopping times such that $0\le S\le \alpha\le\beta\le T$.
If the triple $(Y,M,R)$ is a (semimartingale) solution of {\rm RBSDE}$^{S,T}(\xi,f,L,U)$,
then the triple $(Y,M-M_\alpha,R-R_\alpha)$ is a (semimartingale) solution of  the problem {\rm RBSDE}$^{\alpha,\beta}(Y_\beta,f,L,U)$.
\end{uw}

\begin{uw}
\label{uw.ps.aapx}
A brief inspection of the proofs of Proposition \ref{tw2.0}, Proposition \ref{tw2.1}
shows that Proposition \ref{tw2.0} (ii), Proposition \ref{tw2.1} (ii) hold true if we replace  the constants $n$ in the definition of $f_n$ by any  positive bounded  $\BF$-progressively measurable processes $N^n$ such that $N^n_t\nearrow \infty$ a.s. as $n\rightarrow \infty$ for every
$t\in [\alpha,\beta]$.
\end{uw}

From now on, $\eta$ is a  strictly positive bounded $\BF$-progressively measurable process such that $E\int_0^T\eta_t(S_t-L_t)^-\,dt+E\int_0^T\eta_t(S_t-U_t)^+\,dt<\infty$, where $S$ is the  process appearing in (A1) (for the existence of $\eta$ see the comments after Remark \ref{uw.ps}).

\begin{stw}
\label{tw5.1}
Assume that $f$ satisfies \mbox{\rm(H1)--(H4)} on $[\alpha,\beta]$ with $\mu\le 0$, $\hat \xi$ is an $\FF_\beta$-measurable
integrable random variable such that
\[
\limsup_{a\rightarrow \infty}L_{(\beta\wedge a)\vee\alpha}\le\hat\xi
\]
and $L^+$ is of class \mbox{\rm(D)} on
$[\alpha,\beta]$. Then there exists a unique solution $(Y,M,K)$ of
${\rm{\underline{R}BSDE}}^{\alpha,\beta}(\hat\xi,f,L)$. Moreover, $Y^n_t\nearrow Y_t,\, t\in [\alpha,\beta]$,
where $(Y^n,M^n)$ is a solution to {\rm BSDE}$^{\alpha,\beta}(\xi,f_n)$ with
\[
f_n(t,y)=f(t,y)+n\eta_t(y-L_t)^-.
\]
\end{stw}
\begin{proof}
Without loss of generality we may assume that $L$ is of class (D) (see Remark \ref{uw.kd}).
By \cite[Theorem 2.9]{K:arx}, for every $n\ge 1$ there exists a unique solution $(Y^n,M^n)$ of BSDE$^{\alpha,\beta}(\xi,f_n)$.
By \cite[Proposition 3.1]{KR:JFA}, $Y^n_t\le Y^{n+1}_t$, $t\in [\alpha,\beta]$.
Define $Y$ as $Y_t= \lim_{n\rightarrow \infty} Y^n_t$, $t\in [\alpha,\beta]$.
Observe that $(Y^n,M^n,K^n)$ with $K^n_t=\int_0^tn(Y^n_r-L_r)^-\,dr$ is a solution of ${\rm{\underline{R}BSDE}}^{\alpha,\beta}(\hat\xi,f,L^n)$ with $L^n=L-(Y^n-L)^-$. Let
\[
S_t=S_\alpha+V_t+N_t,\quad t\in [\alpha,\beta],
\]
be the Doob-Meyer decomposition of $S$ ($V$ is a finite variation predictable c\`adl\`ag process with $V_\alpha=0$ and $N$ is a local martingale with $N_\alpha=0$). Let $\tau$ be a stopping time such that $\alpha\le\tau\le\beta$ and let
\[
\sigma_n=\inf\{t\ge\tau: Y^n_t\le L^n_t+\varepsilon\}\wedge \beta.
\]
By the Tanaka-Meyer formula, (H2) and the minimality condition (see (iii) of Definition \ref{df2.2}),
\begin{align*}
(Y^n_\tau-S_\tau)^+&\le E\Big((Y^n_{\sigma_n}-S_{\sigma_n})^+ +\int_\tau^{\sigma_n}\mathbf{1}_{\{Y^n_{r-}>S_{r-}\}}f(r,Y^n_r)\,dr\\
&\quad+\int_\tau^{\sigma_n}\mathbf{1}_{\{Y^n_{r-}>S_{r-}\}}\,dV_r +\int_\tau^{\sigma_n}\mathbf{1}_{\{Y^n_{r-}>S_{r-}\}}\,dK^n_r|\FF_{\tau}\Big)\\
&\le  E\Big((L_{\sigma_n}-S_{\sigma_n})^+\mathbf{1}_{\{\sigma_n<\beta\}}+ (\xi-S_{\beta})^+\mathbf{1}_{\{\sigma_n=\beta\}}\\&
\quad +\int_\tau^{\sigma_n}\mathbf{1}_{\{Y^n_{r-}>S_{r-}\}}f(r,S_r)\,dr +\int_\tau^\beta\,d|V|_r|\FF_{\tau}\Big)+\varepsilon.
\end{align*}
From this inequality, the fact that $L^+,S$ are of class (D), $Y^n\nearrow Y$ and $E\int_\alpha^\beta|f(r,S_r)|\,dr+E\int_\alpha^\beta \,d|V|_r<\infty$ we get that $Y^+$ is of class (D). Since $Y^1\le Y$ we have that $Y$ is of class (D).
Write $\beta_a= (\beta\wedge a)\vee \alpha$. By Proposition \ref{tw2.0} and Remark \ref{uw.ps.aapx} applied on the interval $[\alpha,\beta_a]$, $Y^n_t\nearrow Y^a_t$, $t\in [\alpha,\beta_a]$, where $(Y^a,M^a)$
is a solution of $\underline{{\rm{R}}}$BSDE$^{\alpha,\beta_a}(Y_{\beta_a},f,L)$. Let $M_t= M^a_t$, $t\in [\alpha,\beta_a]$. By uniqueness, $M$ is well defined. We see that $(Y,M)$ is a solution to $\underline{{\rm{R}}}$BSDE$^{\alpha,\beta_a}(Y_{\beta_a},f,L)$ for every $a\ge 0$. What is left is to show that (\ref{eq5.1}) is satisfied.  Since $Y$ is of class (D),  $\sup_{t\in [\alpha,\beta]}|Y_t|$ is finite a.s. Hence, by (H2) and (H4), there exists a chain $\{\tau_k\}$ on $[\alpha,\beta]$  such that
\[
E\int_\alpha^{\tau_k}|f(r,Y_r)|\,dr<\infty,\quad k\ge 1.
\]
Applying now  \cite[Lemma 3.8]{K:arx} on the interval $[\alpha,\tau_k]$, we get
\[
Y_{(\tau_k\wedge a)\vee\alpha}\rightarrow Y_{\tau_k}
\]
as $a\rightarrow\infty$, Since $\{\tau_k\}$ is a chain on $[\alpha,\beta]$, $P(\tau_k<\beta)\rightarrow 0$ as $k\rightarrow\infty$. Consequently,  (\ref{eq5.1}) is satisfied.
\end{proof}


\begin{thebibliography}{32}





\bibitem{A-N}
Alario-Nazaret, M., Lepeltier, J.P. and Marchal, B.: Dynkin games. 
\textit{Lect. Notes Control Inf. Sci.} {\bf 43} (1982) 23--32, Springer-Verlag, Berlin. 


\bibitem{BY}
Bayraktar, E. and  Yao, S.:
Doubly reflected BSDEs with integrable parameters and related Dynkin games.
\textit{Stochastic Process. Appl.} {\bf 125} (2015) 4489--4542. 



\bibitem{Bismut1}
Bismut, J.M.: Sur un probl\`eme de Dynkin.  \textit{Z. Wahrscheinlichkeitstheorie und Verw. Gebiete}  {\bf 39} (1977) 31--53. 

\bibitem{Bismut2}
Bismut, J.M.: Contr\^ole de processus alternants et applications.
\textit{Z. Wahrscheinlichkeitstheorie und Verw. Gebiete}  {\bf 47} (1979) 241--288. 

\bibitem{BL}
Buckdahn, R. and Li, J.: Probabilistic interpretation for systems of Isaacs equations with two reflecting barriers. 
\textit{Nonlinear Differential Equations Appl.} {\bf 16} (2009) 381--420. 



\bibitem{CvitanicKaratzas}
Cvitanic, J. and  Karatzas, I.: Backward Stochastic Differential Equations
with Reflection and Dynkin Games. \textit{Ann.  Probab.} {\bf 19} (1996) 2024--2056.

\bibitem{Dellacherie}
Dellacherie, C.: \textit{Capaciti\'es et processus stochastiques.} Springer-Verlag, Berlin-New York (1972).




\bibitem{DQS}
Dumitrescu, R.,  Quenez, M.-C. and  Sulem, A.: Generalized Dynkin games and doubly reflected
BSDEs with jumps. \textit{Electron. J. Probab.} {\bf 21} (2016) 1--32. 

\bibitem{Dynkin}
Dynkin, E.B.: Game variant of a problem on optimal stopping. \textit{Sov. Math. Dokl.} {\bf 10}  (1969) 270--274. 

\bibitem{DY}
Dynkin, E.B. and Yushkevich, A. A.:  \textit{Theorems and problems in Markov processes.}
Plenum Press, New York (1968).

\bibitem{EHW}
El Asri, B., Hamad\`ene, S., Wang, H.: $L^p$-solutions for doubly reflected backward stochastic differential equations. {\em Stoch. Anal. Appl.} {\bf 29} (2011) 907--932.



\bibitem{EQ}
El Karoui, N. and Quenez, M.C.:
\textit{Non-linear pricing theory and backward stochastic differential equations.}
Financial mathematics (Bressanone, 1996), 191--246, Lecture Notes
in Math. {\bf 1656} Springer, Berlin  (1997). 



\bibitem{GIOQ}
Grigorova, M.,  Imkeller, P.,  Ouknine, Y. and Quenez, M.-C.: 
Doubly reflected BSDEs and $\EE^f$-Dynkin games: beyond the right-continuous case. 
\textit{ Electron. J. Probab.} {\bf 23}  No. 122 (2018) 38 pp. 

\bibitem{GIOOQ}
Grigorova, M.,  Imkeller, P.,   Offen, E., Ouknine, Y. and Quenez, M.-C.: 
Reflected BSDEs when the obstacle is not right-continuous and optimal stopping. 
\textit{Ann. Appl. Probab.} {\bf 27} (2017). 3153--3188.


\bibitem{HH0}
Hamad\`ene, S. and Hassani, M.: 
BSDEs with two reflecting barriers : the general result.
\textit{Probab. Theory Relat. Fields} {\bf 132} (2005)  237--264

\bibitem{HH}
Hamad\`ene, S. and Hassani, M.:  BSDEs with two reflecting
barriers driven by a Brownian and a Poisson noise and related
Dynkin game. \textit{Electron. J. Probab.} {\bf 11} (2006) 121--145. 



\bibitem{HHO}
Hamad\`ene, S., Hassani, M. and Ouknine, Y.:
Backward SDEs with two rcll reflecting barriers without Mokobodski's hypothesis.
\textit{Bull. Sci. Math.} {\bf 134} (2010) 874--899. 

\bibitem{HL}
Hamad\`ene, S. and Lepeltier, J.-P.: Reflected BSDEs and mixed game problem. \textit{Stochastic
Process. Appl}. {\bf 85} (2000) 177--188. 

\bibitem{HO}
Hamad\`ene, S. and Ouknine, Y.: 
Reflected backward SDEs with general jumps.  \textit{Theory Probab. Appl}. {\bf 60} (2016) 263--280. 

\bibitem{HW}
Hamad\`ene, S. and Wang, H.: BSDEs with two RCLL reflecting
obstacles driven by Brownian motion and Poisson measure and
related mixed zero-sum game.  \textit{Stochastic Process.  Appl.} {\bf
119}  (2009)  2881--2912. 





\bibitem{Kl:EJP}
Klimsiak, T.:  Reflected BSDEs with monotone generator.
\textit{Electron. J. Probab.} {\bf 17} (2012) 25 pp.  

\bibitem{K:PA}
Klimsiak, T.:  Cauchy problem for semilinear parabolic equation with
time-dependent obstacles: a BSDEs approach. \textit{Potential
Anal.} {\bf 39}  (2013) 99--140. 

\bibitem{Kl:BSM}
Klimsiak, T.:  BSDEs with monotone generator and two
irregular reflecting barriers. \textit{Bull. Sci. Math.} {\bf 137} (2013) 268--321.

\bibitem{K:SPA2}
Klimsiak, T.:  Reflected BSDEs on filtered probability spaces.
\textit{Stochastic Process. Appl.} {\bf 125} (2015) 4204--4241. 



\bibitem{K:JEE}
Klimsiak, T.:  Obstacle problem for evolution equations involving measure data and operator corresponding to semi-Dirichlet form. 
 \textit{J. Evol. Equ.} {\bf 18}  (2018) 681--713. 

\bibitem{K:arx}
Klimsiak, T.:  Systems of quasi-variational inequalities related to
the switching problem.  \textit{ Stochastic Process. Appl.} {\bf 129} (2019) 1259--1286.


\bibitem{KR:JFA}
Klimsiak, T. and Rozkosz, A.:  Dirichlet forms and semilinear elliptic
equations with measure data. \textit{J. Funct. Anal.} {\bf 265} (2013) 890--925.

\bibitem{KR:JEE}
Klimsiak, T. and Rozkosz, A.:   Obstacle problem for semilinear parabolic equations with measure data. 
{\textit J. Evol. Equ.} {\bf 15} (2015) 457--491.


\bibitem{KRS:SPA}
Klimsiak, T.,  Rzymowski, M. and  S\l omi\'nski, L.:  Reflected BSDEs with regulated trajectories. 
\textit{Stochastic Process. Appl.} {\bf 129} (2019)  1153--1184. 


\bibitem{KQC}
Kobylanski, M.,  Quenez, M.-C. and   Campagnolle, M.:   Dynkin games in a general framework. 
\textit{Stochastics} {\bf 86} (2014)  304--329.



\bibitem{LepeltierMatoussiXu}
Lepeltier, J.-P., Matoussi, A. and  Xu, M.:   Reflected
bakcward stochastic differential equations under monotonicity and
general increasing growth conditions. \textit{Adv. in Appl. Probab.}
{\bf 37}  (2005) 134--159. 

\bibitem{LM}
Lepeltier, J.-P. and  Maingueneau, M.A.:   Le jeu de Dynkin en theorie generale sans
l'hypothese de Mokobodski. \textit{Stochastics} {\bf 13} (1984) 25--44. 

\bibitem{Lin}
Q. Lin.:   Nash Equilibrium Payoffs for Stochastic Differential Games with two Reflecting Barriers. 
\textit{Advances in Applied Probability} {\bf 47} (2015) 355--377. 



\bibitem{Morimoto}
Morimoto, H.:   Dynkin games and martingale methods.
\textit{Stochastics} {\bf 13} (1984)  213--228.



\bibitem{Peng1}
Peng, S.:  \textit{Backward SDE and related g-expectations}. 
In: Backward Stochastic Differential Equations, N. El Karoui and L. Mazliak, Eds.,
vol. 364 of Pitman Research Notes in Mathematics Series, Longman, Harlow, UK (1997). 

\bibitem{Peng2}
Peng, S.:  \textit{Nonlinear expectations, nonlinear evaluations and risk measures}. Lecture
Notes in Math. {\bf 1858}, Springer, Berlin (2004). 


\bibitem{QS}
Quenez M.-C. and  Sulem, A.:  Reflected BSDEs and robust optimal stopping for dynamic risk
measures with jumps. \textit{Stochastic Process. Appl.} {\bf 124} (2014) 3031--3054. 

\bibitem{Robin}
Robin, M.:  \textit{ Controle impulsionnel des processus de Markov}. Thesis, University of Paris IX (1978).

\bibitem{RS:SPA}
Rozkosz, A. and   S\l omi\'nski, L.:  $L^p$ solutions of
reflected BSDEs under monotonicity condition. \textit{Stochastic
Process. Appl.} {\bf 122} (2012) 3875--3900. 

\bibitem{RS:EJP}
Rozkosz, A. and  S\l omi\'nski, L.:   Stochastic representation of entropy solutions of semilinear elliptic obstacle problems with measure data. 
\textit{Electron. J. Probab.} {\bf 17}, no. 40 (2012) pp. 27. 



\bibitem{Stettner}
Stettner, \L.:  Zero-sum Markov games with stopping and impulsive strategies. \textit{Appl.
Math. Optim.} {\bf 9} (1982) 1--24. 

\bibitem{Stettner3}
Stettner, \L.:  On closedness of general zero-sum stopping game.
\textit{Bull. Polish Acad. Sci.} {\bf 32}  (1984) 351--361. 

\bibitem{Stettner1}
Stettner, \L.: Penalty method for finite horizon stopping problems. \textit{SIAM J. Control Optim. }{\bf 49} (2011) 1078--1099.

\bibitem{SZ}
Stettner, \L. and  Zabczyk, J.: Strong envelopes of stochastic processes and a penalty method.
\textit{Stochastics} {\bf 4} (1980/1981) 267--280. 

\bibitem{T}
Topolewski, M.:  Reflected BSDEs with general filtration and two completely
separated barriers. \textit{Probab. Math. Statist.} {\bf 39} (2019) 199--218. 



\bibitem{Zabczyk}
Zabczyk, J.:  Stopping games for symmetric Markov processes.
\textit{Probab. Math. Statist.} {\bf 4}  (1984)  185--196. 

\end{thebibliography}
\end{document}